\documentclass{article}

\usepackage[T1]{fontenc}

\usepackage{geometry}
\usepackage{csquotes}
\usepackage{amsmath,amsfonts,amsthm}
\usepackage{dsfont}

\usepackage{graphicx} 
\usepackage{subcaption}

\usepackage{hyperref}
\usepackage{orcidlink}

\usepackage{pgfplots}
\pgfplotsset{compat=1.18}
\usepackage{tikz}
\usetikzlibrary{arrows.meta,bending}

\usepackage{standalone}

\usepackage{enumitem}
\usepackage{booktabs}   

\usepackage{stackengine}
\usepackage{accents}
\usepackage{bbm}

\usepackage[left]{lineno}
\usepackage{soul}

\usepackage
[
backend=biber,
style=alphabetic,
maxnames=5, 
giveninits=true, 
isbn=false,
backref=true
]
{biblatex}

\DeclareSourcemap{
	\maps[datatype=bibtex]{
		\map[overwrite]{
			\step[fieldsource=doi, final]
			\step[fieldset=url, null]
			\step[fieldset=eprint, null]
		}
        \map{
        \step[fieldset=date, null]
        \step[fieldset=isbn, null]
        }
	}
}

\usepackage{todonotes}

\newcommand{\jshsr}[1]{{\textcolor{black}{#1}}}
\newcommand{\jshsb}[1]{{\textcolor{black}{#1}}}
\newcommand{\ck}[1]{ {\color{green!70!black}#1} }

\def\bu{\mathbf{u}}
\def\bud{\widetilde{\mathbf{u}}}
\def\bua{\widehat{\mathbf{u}}}
\def\bv{\mathbf{v}}
\def\bvd{\widetilde{\mathbf{v}}}
\def\bva{\widehat{\mathbf{v}}}
\def\bw{\mathbf{w}}

\def\bz{\mathbf{z}}
\def\bf{\mathbf{f}}
\def\bg{\mathbf{g}}
\def\bn{\mathbf{n}}
\def\bW{\boldsymbol{W}}
\def\bL{\boldsymbol{L}}
\def\bH{\boldsymbol{H}}
\def\bF{\mathbf{F}}
\def\bG{\mathbf{G}}
\def\bV{\boldsymbol{V}}

\def\bpsi{\boldsymbol{\psi}}
\def\du{\,\mathrm{d}}
\newcommand{\dx}{\du x}

\DeclareMathOperator{\dive}{div}
\def\one{\chi}
\def\pd{\widetilde{p}}
\def\pa{\widehat{p}}
\def\pid{\widetilde{\pi}}
\def\pia{\widehat{\pi}}
\def\ST{\mathcal{S}_T}
\def\Ss{\mathcal{S}_s}
\def\alphat{\beta}
\DeclareMathOperator{\ws}{\mathrel{\ensurestackMath{\stackon[1pt]{\rightharpoonup}{\scriptstyle\ast}}}}
\DeclareMathOperator{\wPhi}{\mathrel{\ensurestackMath{\stackon[1pt]{\rightharpoonup}{\Phi}}}}

\theoremstyle{plain}
\newtheorem{theorem}{Theorem}[section]
\newtheorem{proposition}[theorem]{Proposition}
\newtheorem{lemma}[theorem]{Lemma}
\newtheorem{claim}[theorem]{Claim}

\theoremstyle{definition}
\newtheorem{definition}[theorem]{Definition}
\newtheorem{assumption}[theorem]{Assumption}

\theoremstyle{remark}
\newtheorem{remark}[theorem]{Remark}

\title{Long-time behavior of solutions to a fluid dynamic shape optimization problem via phase-field method
}
\author{Michael Hinze%
	\footnote{Mathematical Institute, University of Koblenz, Germany}
	\orcidlink{0000-0001-9688-0150},\,%
	\and Christian Kahle%
	\footnotemark[1]
	\orcidlink{0000-0002-3514-5512},\,%
	\and John Sebastian H. Simon%
	\footnotemark[1]
	\orcidlink{0000-0002-7711-9582}
}
\date{\today}

\begin{document}
	
	\maketitle

	\begin{abstract}
		We investigate the long time behavior of solutions to a shape and topology optimization problem with respect to the time-dependent Navier--Stokes equations. The sought topology is represented by a stationary phase-field that represents a smooth indicator function. The fluid equations are approximated by a porous media approach and are time-dependent. In the latter aspect, the considered problem formulation extends earlier works.

		We prove that if the time horizon tends to infinity, minima of the time-dependent problem converge towards minima of the corresponding stationary problem. To do so, a convergence rate, with respect to the time horizon, of the values of the objective functional is analytically derived. This allowed us to prove that the solution to the time-dependent problem converges to a phase-field, as the time horizon goes to infinity, which is proven to be a minimizer for the stationary problem. We validate our results by numerical investigation. 
		
			\medskip
		\noindent\textbf{Keywords.} Navier--Stokes equations, shape optimization, phase-field method, long-time behavior
		
		\noindent\textbf{MSC Codes.} 35Q93, 35Q30, 76D55, 35R35
	\end{abstract}

	\section{Introduction}
	We study the long-time behavior of solutions to shape and topology optimization problems governed by the time-dependent Navier–Stokes equations. The optimization problem is formulated on a finite time horizon, whereas the shape variable is assumed to be time independent. Our primary interest lies in the convergence properties of the optimal shapes as the time horizon tends to infinity.
	
	Shapes are represented implicitly by a \jshsb{stationary} phase-field variable $\varphi$ defined on a fixed hold-all domain $\Omega\subset \mathbb{R}^2$. Within this framework, we consider the following optimization problem:
	\begin{align*}
		\min_{\bu,\varphi} J_T(\bu,\varphi) &= \frac{1}{T}\left[\int_0^T\int_\Omega \left( \frac{\varphi + 1}{2} \right) |\bu-\bu_d|^2 \one_\omega \du x\du t + \int_0^T\int_\Omega \alphat_\varepsilon(\varphi)|\bu|^2\du x\du t 
		\right]  + \gamma \mathbb{E}_{\varepsilon}(\varphi)\\
		\text{ subject to }& \nonumber\\
		&
		\begin{aligned}
			\partial_t \bu + \alpha_\varepsilon(\varphi)\bu - \mu\Delta \bu + (\bu\cdot\nabla)\bu + \nabla p & = \bf &&\text{ in }\Omega_T:= \Omega\times (0,T),\\
			\dive\bu & = 0 && \text{ in }\Omega_T,\\
			\bu & = \bg &&\text{ on }\partial \Omega_T = \partial \Omega \times (0,T),\\
			\bu(0) & = \bu_0 &&\text{ in }\Omega,
		\end{aligned}
	\end{align*}
    \jshsr{where $(\bu,p): \Omega \to \mathbb{R}^2\times \mathbb{R}$ corresponds to the velocity-pressure pair of the fluid, $\mu$ is the kinematic viscosity, $\bf$ denotes an external force acting on the fluid, $\bg$ is a given Dirichlet data acting on the boundary, and $\bu_0$ is the initial velocity of the fluid. The function $\chi_\omega$ is the characteristic function for the subdomain $\omega\subset \Omega$ of observation, and $\alpha_\varepsilon$ and $\beta_\varepsilon$ are Brinkman-type interpolating functions which model the influence of the shape on the fluid system.
	This formulation can be interpreted as a classical $L^2$ tracking-type optimization problem for a target velocity field $\bu_d$, augmented by a perimeter regularization term expressed through the Ginzburg--Landau energy functional $\mathbb{E}_{\varepsilon}$. Since the phase-field variable $\varphi$ acts as coefficient in the state equation, the shape/topology optimization problem is rendered into a minimization problem with PDE constraints, where now  $\varphi$ acts as control. This enables us to invoke the well developed machinery from PDE constrained optimization, found for example in \cite{hinze2009}, for the analytical and numerical treatment of this 
    optimization problem.}
	For a detailed description of the model and assumptions, we refer to Section~\ref{sec:setup}.

  \jshsr{ We note that in fluid dynamic optimization problems, frequently channel flows are  considered which call for imposing boundary conditions that model inflow and/or outflow. Usually, one considers the so-called do-nothing \cite{heywood1996artificial} or the directional do-nothing boundary condition \cite{braack2014directional}. However, existence (do-nothing) and/or uniqueness (directional do nothing) results for weak solutions are lacking, rendering it not possible to work with reduced cost functionals, which we extensively use for the analysis in Section \ref{sec:optProblems}. Moreover, with the directional do-nothing conditions we are also faced with non-smoothness issues, which affect the differentiability of the control to state operator.
  }
	
	\paragraph*{Contribution}
	The highlight of this article is the inquiry of the asymptotic behavior of the optimal values of $J_T$. To do so, we compare the evaluations of $J_T$ on their global optimizers against the optimal value of a corresponding problem with a stationary flow field, see Theorem~\ref{thm:asymp}. 
	As a consequence, we show that global minimizers of $J_T$ converge to a global minimizer of the stationary problem as $T\to \infty$, see Theorem~\ref{thm:phiInftyIsMin}.

	In addition, by incorporating time-dependent fluid dynamics, we extend earlier results from the stationary setting considered e.g. in \cite{garcke2015} to the fully time-dependent case.

	\paragraph{Literature Review}

	Foundational contributions in the direction of optimization involving fluids were made in \cite{abergel1990} and in \cite{fursikov1983}, where first-order optimality conditions for optimal control problems constrained by the Navier--Stokes equations were derived. These works laid the groundwork for gradient-based algorithms for approximating optimal controls in time-dependent fluid systems. Since then, a substantial body of literature has developed around optimization problems constrained by fluid dynamics, see, e.g., \cite{casas2012,casas2007,deckelnick2003,deckelnick2004,gunzburger1999,gunzburger2000}.  
	
	A particular class of such problems concerns the optimization of the shape of the fluid domain, commonly referred to as \emph{shape optimization problems}. One of the earliest and most influential works in this direction is \cite{pironneau1973}, establishing first-order necessary conditions for drag minimization constrained by the Stokes equations. This seminal result inspired subsequent investigations, including \cite{bello1992,bello1997}, where the differentiability of the drag functional with respect to domain variations were analyzed. More recent works, such as \cite{gao2008,iwata2010,kasumba2013,schmidt2010}, have further advanced this line of research by considering numerical implementations.  
	
	A defining feature of shape optimization problems is the variability of the underlying domain, which introduces both analytical and numerical difficulties. Analytically, the variation of the domain affects the functional spaces in which solutions are defined, thereby complicating convergence analyses. Numerically, modifying the domain during the solution process often necessitates specialized adaptations, making the problem more intricate and less tractable. To address these challenges, several authors have proposed reformulations in which the computational domain is fixed while the geometry of interest is embedded within it, i.e., 
	the geometry is implicitly described through an auxiliary function. Among such approaches, the phase-field method has gained considerable attention. A series of works \cite{garcke2015,garcke2016,garcke2016b,garcke2018} applied the phase-field method to shape optimization problems governed by the stationary Navier--Stokes equations, demonstrating in particular its ability to accommodate topological changes in the fictitious domain.

	A particular property that the previously mentioned works dealing with the phase-field approach to shape optimization is that the fluid is governed by the stationary Navier--Stokes equation. Although these works provide vital information for solving topology and shape optimization problems, the nonlinear nature of fluid dynamics gives rise to phenomena only observable in the unsteady case. A few examples of such are vortex shedding, transient separation and wakes \cite{wu1998}. We also mention the case where the fluid equations do not possess a stable equilibrium point \cite{ngom2015,suri2018}.

	Finally, particular attention is drawn to the role of the design function $\varphi$ in the state equation, that appears as a control in the coefficient. In the typical case where $\alpha_\varepsilon$ depends linearly on $\varphi$, the problem resembles bilinear optimal control problems, 
	which have been applied, for example, to the modeling of ecological phenomena \cite{mazari2022,mazari2023} and of the Bose-Einstein condensates and non-relativistic bosonic atoms and molecules \cite{hintermuller2013,feng2013} .

	\paragraph{Motivation}

	The main advantage of using phase fields is that the shapes and the inherent topology of the domain are quantified by a design function. \jshsr{This allows us to naturally handle topological changes during the iterative process of approximating the optimal solution, meaning the initial guess does not need to share the topology of the final optimal shape.}
	
	Furthermore, we emphasize that the phase-field method, together with the present investigation, provides a pathway toward understanding the turnpike property in shape optimization. Essentially, the turnpike property in optimization is the phenomenon that the optimal trajectory spends a long-time near a steady-state, see, e.g., \cite{zaslavski2014}.

	The phase-field approach offers a  mathematically and numerically tractable representation of shape and topology. 
	Specifically, the current formulation leads to a first-order optimality system that directly links the optimal shape and optimal state. This does not only facilitate numerical solution but also enables a systematic quantification of shapes and topology, which may support future proofs of the turnpike property in shape optimization. 
	
	With respect to the convergence of shape with respect to time we are essentially building up on the works of \cite{simon2023} and \cite{lance2019}. \jshsr{In the former, the shape optimization problem is posed in the \emph{sharp} setting where the analysis is based on the topology of the characteristic functions of the admissible domains, and the numerics utilized a sensitivity analysis enabling the author to design a steepest-descent algorithm that searches for deformation fields that alters a pre-defined domain and reduces the objective functional.} The latter attempted to establish the turnpike property for a heat-equation-constrained shape optimization problem and achieved numerical evidence. 

    \jshsr{We further point out that even though the work in \cite{simon2023} and this manuscript address the same physical goal, i.e., investigating the long-time behavior of shape optimization problems involving the Navier--Stokes equations, the current investigation approaches the problem with a fully  mathematically tractable approach through the phase-field. Our utilization of the Ginzburg--Landau energy functional equips our analysis with a natural topology for the quantification of the domain. On the other hand, the proofs in \cite{simon2023} utilize restrictive assumptions on the admissible domains to allow to work with the so-called cone property, which in consequence provides an avenue to compactness arguments.  We also underline the consistency between the analysis and the numerical implementation that the phase-field method afforded us, i.e., both the analysis and the numerics revolve around looking for the minimizing phase-field variable $\varphi$.  }

	The phase-field formulation resembles a bilinear optimal control problem, in which the control and state appear multiplicatively in the governing equation, as seen in Section~\ref{sec:setup}. The turnpike property for such bilinear problems is an active area of research, see, e.g., \cite{mazari2022b}. Our work will also contribute to investigating the turnpike property in bilinear optimal control problems constrained by the Navier–Stokes equations.

	Additionally, for this work we extend results for topology optimization using a phase field formulation for the stationary Navier--Stokes equation from \cite{garcke2015,garcke2016,garcke2016b,GHHKL2016-ifb,garcke2018} to the setting with  \emph{time-dependent} Navier--Stokes equations.
	We analyze both the time-dependent optimization problem and its long-time behavior, thereby establishing a methodological framework with broader applicability. Notably, the analytical tools developed herein are not restricted to the specific problem at hand; they may also be adapted to alternative objective functionals or to other physical models, such as those studied in \cite{garcke2022,garcke2024,auricchio2024}, or to simpler systems such as the heat and Poisson equations.

	\paragraph*{Outline}
	In Section~\ref{sec:setup} we introduce the problem that we investigate.
	Thereafter we collect preliminaries including standing assumptions in Section~\ref{sec:preliminaries}. 
	In Section~\ref{sec:solGovEq} we summarize and investigate both the existence and regularity of solutions to the governing fluid equations in the stationary and the time-dependent  setting and introduce and investigate the stationary and time-dependent optimization problem in Section~\ref{sec:optProblems}. The main results on long time behavior of global minimizers are derived in Section~\ref{sec:LongTimeBehavior}.
	The results are validated in Section~\ref{sec:numerics} through a numerical example.


	\section{Setup and model}
	\label{sec:setup}
	We describe the shape and topology optimization problem using a phase field approach, following e.g. in \cite{garcke2015,garcke2018} and we refer for a decent introduction of the model to this literature. We consider the problem with both time-dependent and stationary state equations.

	Let $\Omega$ denote a hold-all domain that is supposed to be divided into a fluid domain and an obstacle.
	To utilize the so-called phase field method, we consider a smooth indicator function that serves as a design function, $\varphi : \Omega \to [-1,1]$ such that the set $\{x\in\Omega: \varphi(x) = 1\}$ indicates the fluid domain, while $\{x\in\Omega: \varphi(x) = -1\}$ approximates the obstacle with a porous medium. Furthermore, the design function is assumed to change its value rapidly but continuously from $+1$ to $-1$, which leads to an interface  of thickness proportional to $\varepsilon>0$ at the boundary of the fluid domain where $|\varphi| < 1$. This thickness $\varepsilon$ is a given modeling parameter.
	In Figure~\ref{fig:setup:sketch} a sketch of the situation is drawn.

	The introduction of the design function $\varphi$, which implicitly describes the fluid domain, calls us to introduce a model for the fluid equations that are now posed on the hold-all domain $\Omega$.
	For this we consider a porous media approximation following \cite{article:BorrvallPetersson03} by assuming that the obstacle has a very small permeability $(\overline{\alpha}_\varepsilon)^{-1}\ll 1$.
	Moreover we introduce a function $\alpha_\varepsilon:\mathbb{R}\to \mathbb{R}$ satisfying $\alpha_\varepsilon(1) = 0$ and $\alpha_\varepsilon(-1) = \overline{\alpha}_\varepsilon$ to interpolate between the fluid domain and the porous media.
	Consequently, the state equations are defined in the whole domain $\Omega$
	and we respectively use the following state equations governing the time-dependent velocity-pressure pair $(\bu,p)$ and the stationary velocity-pressure pair $(\bv,\pi)$
	\begin{align}
		\label{poreq:timeNS}
		&\left\{
		\begin{aligned}
			\partial_t \bu + \alpha_\varepsilon(\varphi)\bu - \mu\Delta \bu + (\bu\cdot\nabla)\bu + \nabla p & = \bf &&\text{ in }\Omega_T:= \Omega\times (0,T),\\
			\dive\bu & = 0 && \text{ in }\Omega_T,\\
			\bu & = \bg &&\text{ on }\partial \Omega_T = \partial \Omega \times (0,T),\\
			\bu(0) & = \bu_0 &&\text{ in }\Omega,
		\end{aligned}
		\right.\\ 
		\text{and } \hspace{1.5cm}& \nonumber \label{poreq:statNS}\\
		&\left\{
		\begin{aligned}
			\alpha_\varepsilon(\varphi)\bv - \mu\Delta \bv + (\bv\cdot\nabla)\bv + \nabla \pi & = \bf_s &&\text{ in }\Omega,\\
			\dive\bv & = 0 && \text{ in }\Omega,\\
			\bv & = \bg_s &&\text{ on }\partial \Omega.
		\end{aligned}
		\right.
	\end{align}
	Here $\bf,\bg$ and $\bf_s,\bg_s$ are given external volume force and Dirichlet data for the time-dependent and stationary equations, respectively, $\bu_0$ is a fixed initial data, and $\mu>0$ is the kinematic viscosity. 
	
	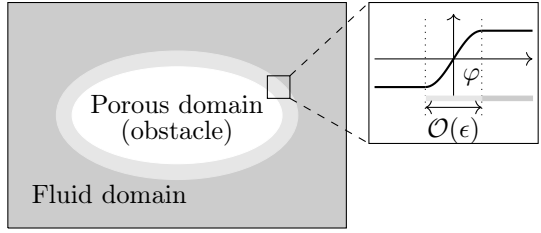
\begin{figure}
		\centering
		\begin{tikzpicture}[scale=1.5]

  \draw[fill=black!20](0.0,0.0) -- (0.0,2.0) -- (3.0,2.0) -- (3.0,0.0) -- (0.0,0.0);

  \draw[fill,color=black!10] (1.5,1.0) ellipse(1.07 and 0.57);
  \draw[fill,color=black!00] (1.5,1.0) ellipse(0.93 and 0.43);

      \draw (0.9,0.3) node {Fluid domain};
   \draw (1.5,1.1) node {Porous domain};
   \draw (1.5,0.85) node {(obstacle)};
  


  \begin{scope}[xshift=3.2cm,yshift=0.75cm]

    
    \draw (-0.9,0.4) rectangle ++(0.20,0.20);

    \draw[dashed] (-0.7,0.4) -- (0,0);
    \draw[dashed] (-0.7,0.6) -- (0,1.25);

    \draw (0,0) rectangle(1.5,1.25);

    \draw[->] (0.05,0.75) -- (1.45,0.75); 
    \draw[->] (0.75,0.40) -- (0.75,1.15); 

    \draw[dotted] (0.5,0.25) -- (0.5,1.15);
    \draw[dotted] (1.0,0.25) -- (1.0,1.15);


    \colorlet{clphi}{black!100}
    \draw[thick,color=clphi,domain=-1:1] plot (0.75 - \x/4,{0.75-0.25*(sin(pi/2*\x r))});
    \draw[thick,color=clphi,solid] (0.05,0.5) -- (0.5,0.5);
    \draw[thick,color=clphi,solid] (1.0,1.0) -- (1.45,1.0);

    \draw (0.75,0.75) node[anchor=north west]{\color{clphi}$\varphi$};

    \draw[fill,color=black!20] (1.0,0.38) rectangle ++(0.45,0.04);
    \draw[fill,color=black!10] (0.5,0.38) rectangle ++(0.5,0.04);
    \draw[fill,color=black!00](0.05,0.38) rectangle ++(0.45,0.04);
    \draw[<->] (0.5,0.32)--(1.0,0.32);
    \draw (0.75,0.32) node[anchor=north] {$\mathcal O (\epsilon)$};

  \end{scope}

\end{tikzpicture}
		\caption{The geometrical setup using a phase field function $\varphi$.}
		\label{fig:setup:sketch}
	\end{figure}
	
	The objective functionals that we  consider is of tracking-type within the fluid domain, where we track a given stationary target velocity field $\bu_d$ in the $L^2$-norm within some given observation domain $\omega$, for which $\chi_\omega$ denotes its indicator function.
	The corresponding integral is extended to $\Omega$ by using   $\frac{\varphi+1}{2}$  as indicator function for the fluid domain. 
	
	For the minimizers $\varphi$ of the objective functionals to have the form of a phase-field, we shall consider the Ginzburg--Landau energy functional
	\begin{align*}
		\mathbb{E}_\varepsilon(\varphi) = \frac{1}{2c_0}\left( \int_\Omega \frac{\varepsilon}{2}|\nabla\varphi|^2 + \frac{1}{\varepsilon}\Psi(\varphi) \du x \right),
	\end{align*}
	with some regularization constant $\gamma>0$ to the objective.
	Here $\Psi:[-1,1]\to\mathbb{R}$ is a given potential with equal minima at $-1$ and $1$, and $c_0>0$ is a constant dependent only on $\Psi$ in the sense that
	\begin{align*}
		c_0 = \frac{1}{2}\int_{-1}^1 \sqrt{2\Psi(s)}\du s.
	\end{align*}
	Furthermore, from the definition of $\mathbb{E}_\varepsilon$ we see that admissible design functions belong to
	\begin{align*}
		\Phi_{ad} := \{\varphi\in H^1(\Omega) : |\varphi (x)| \le 1 \text{ a.e. in }\Omega\}
	\end{align*}
	and  by definition, we have for any $\varphi\in\Phi_{ad}$
	\begin{align}\label{bound:Phi}
		\|\varphi\|_{L^\infty} \le 1.
	\end{align}

	From the introduction of the porous media approximation, we note that the contribution of fluid velocity on the interface may not be at all negligible. For this reason, we also consider a penalization for this contribution on the objective functional. Let us consider a function ${\beta}_\varepsilon : \mathbb{R}\to \mathbb{R}$ possessing the same properties as $\alpha_{\varepsilon}$ above, i.e., there exists $\overline{\beta}_\varepsilon\in\mathbb{R}$ such that $(\overline{\beta}_\varepsilon)^{-1}\ll 1 $  from which we define $\beta_\varepsilon$ as $\beta_\varepsilon(1) = 0$ and $\beta_{\varepsilon}(-1) = \overline{\beta}_\varepsilon$. 
	
	More precisely, we respectively consider the following objective functions $J_T$, for the time-dependent setting, and $J_s$, for the stationary setting:
	\begin{align}
		&  J_T(\bu,\varphi) = \frac{1}{T}\left[\int_0^T\int_\Omega \left( \frac{\varphi + 1}{2} \right) |\bu-\bu_d|^2 \one_\omega \du x\du t + \int_0^T\int_\Omega \alphat_\varepsilon(\varphi)|\bu|^2\du x\du t 
		\right]  + \gamma \mathbb{E}_{\varepsilon}(\varphi)\label{objphasetime},\\
		&  J_s(\bv,\varphi) = \int_\Omega \left( \frac{\varphi + 1}{2} \right) |\bv-\bu_d|^2\one_\omega \du x  + \int_\Omega \alphat_\varepsilon(\varphi)|\bv|^2\du x 
		+ \gamma \mathbb{E}_{\varepsilon}(\varphi). \label{objphasestat}
	\end{align}
	The main result of this work is the proof of the following results, \jshsr{which are stated in Theorem~\ref{thm:asymp} and Theorem~\ref{thm:phiInftyIsMin}.}
    \jshsr{
	\begin{claim}
		Let $\varphi^T$ and $\varphi^s$ be global minimizers of the minimization problems associated with the cost functionals \eqref{objphasetime} and \eqref{objphasestat}, respectively, and suppose that some meaningful assumptions are imposed on the data.
        \begin{enumerate}
            \item There is $C>0$, independent of $T$, such that
            \begin{align*}
			|J_T(\bu(\varphi^T),\varphi^T) - J_s(\bv(\varphi^s),\varphi^s)| \le 
            C \left( \frac{1}{T} + \frac{1}{\sqrt{T}}\right).
		\end{align*}
        \item The family of solutions $(\varphi^T)\subset \Phi_{ad}$ converges, up to a subsequence, to an element $\varphi^\infty\in\Phi_{ad}$ which is a minimizer of an optimization problem corresponding to the objective functional \eqref{objphasestat}.
        \end{enumerate}
	\end{claim}
    }

    \jshsr{
	\begin{remark}
	    As a primary consequence of our analysis, we establish that the stationary shape and topology optimization problem is the asymptotic limit of the transient formulation. Physically, this indicates that the time-averaged performance of a stationary phase-field evaluated over a long time horizon aligns with its steady-state performance. The inequality in the claim above demonstrates that the optimal shape for the evolutionary system is asymptotically governed by the stationary physics. From a practical standpoint, this provides a rigorous justification for using the stationary solution, not only as a consistent approximation of the long-time behavior but also as a robust initial guess for high-fidelity, time-dependent shape optimization.
	\end{remark}
    }
	
	\begin{remark}
		Problems~\eqref{objphasetime} and \eqref{objphasestat} are approximations of the corresponding formal sharp problems
		\begin{align*}
			\min_{B}\, J_T(\bu,B) &:= \frac{1}{T}\int_0^T\int_E|\bu-\bu_d|^2 \one_\omega \du x\du t + \gamma \mathrm{P}_\Omega(B)
			\\
			\text{ subject to }& \nonumber\\
			&
			\left\{
			\begin{aligned}
				\partial_t \bu - \mu\Delta \bu + (\bu\cdot\nabla)\bu + \nabla p & = \bf &&\text{ in }E_T:= E\times (0,T),\\
				\dive\bu & = 0 && \text{ in }E_T,\\
				\bu & = 0 &&\text{ on }\partial B_T:= \partial B\times(0,T),\\
				\bu & = \bg &&\text{ on }\partial\Omega_T:= \partial\Omega\times(0,T),\\
				\bu(0) & = \bu_0 &&\text{ in }E,
			\end{aligned}
			\right.
		\end{align*}
		and
		\begin{align*}
			\min_{B}\, J_s(\bv,B) &:= \int_E |\bv-\bu_d|^2 \one_\omega \du x + \gamma \mathrm{P}_\Omega(B), \\
			\text{ subject to } &\nonumber\\
			&
			\left\{
			\begin{aligned}
				- \mu\Delta \bv + (\bv\cdot\nabla)\bv + \nabla \pi & = \bf_s &&\text{ in }E,\\
				\dive\bv & = 0 && \text{ in }E,\\
				\bv & = 0 &&\text{ on }\partial B,\\
				\bv & = \bg_s &&\text{ on }\partial
				\Omega,
			\end{aligned}
			\right.
		\end{align*} 
		where $B$ denotes the obstacle, $E$ denotes the fluid domain and $P_\Omega(B)$ denotes the perimeter of $B$.
		Formal matched  asymptotics for $\varepsilon \to 0$ are carried out for the stationary setting in \cite{GHHKL2016-ifb}.
		Especially,   $\mathbb{E}_\varepsilon(\varphi)$ $\Gamma$-convergences to the perimeter functional $P_\Omega(B)$    as $\varepsilon\to 0$, see \cite{Modica_GradientTheoryPhaseTransition}.
	\end{remark}


	\section{Preliminaries and standing assumptions}
	\label{sec:preliminaries}
	Let $X$ be a Banach space with the norm $\| \cdot\|_X$, we denote its dual as $X^*$ and we write $\langle x^*,x\rangle_X $ for the dual pairing of $x\in X$ and $x^*\in X^*$. 
	We denote by $L^p(D)$ the space containing Lebesgue measurable and $p$-integrable functions on a measurable domain $D$. We denote by $\|\cdot\|_{L^p}$ the norm in $L^p(D)$, whenever the domain $D$ is known. The $L^2$-inner product in a domain $D$ shall be denoted as $(\cdot,\cdot)_D$. 
	
	Let $p\ge 1$ and $m\ge 0$, the Sobolev spaces in a domain $\Omega$ will be denoted as $W^{m,p}(\Omega)$, with $H^m(\Omega) = W^{m,2}(\Omega)$. The norm in $W^{m,p}(\Omega)$ is denoted by $\|\cdot\|_{W^{m,p}}$. 
	We also consider zero trace Sobolev spaces which we denote as $W^{1,p}_0(\Omega)$ or $H^1_0(\Omega)$ for $p=2$. 
	We write in bold letter, for instance $\bL^p(\Omega)$ and $\bW^{m,p}(\Omega)$, to emphasize spaces consisting of vector-valued functions.

	To take into account the incompressibility of the fluid we consider the following solenoidal spaces
	\begin{align*}
		&\bV := \{ \bpsi\in \bH_0^{1}(\Omega): \dive \bpsi = 0 \text{ in }\Omega \}\text{ and }\\
		&\bH := \{ \bpsi\in\bL^2(\Omega) : \dive\bpsi = 0 \text{ in }L^2(\Omega), \bpsi\cdot\bn = 0 \text{ on }\partial\Omega \}.
	\end{align*}
	The spaces $\bV$ and $\bH$ are respectively endowed with the norm
	\begin{align*}
		\|\bpsi\|_{\bV} = \| \nabla\bpsi\|_{\bL^2}\text{ and } \|\bpsi\|_{\bH} = \| \bpsi\|_{\bL^2}.
	\end{align*}
	The pair $(\bV,\bH)$ is known to satisfy the Gelfand triple, i.e., the embeddings $\bV\hookrightarrow \bH \hookrightarrow \bV^*$ are dense, continuous and compact. In fact, the compactness follows from Rellich--Kondrachov embedding and Schauder theorems.

	We recall that the norm $\|\cdot\|_{\bV}$ is equivalent to the norm in $\bH^1(\Omega)$ due to Poincar{\'e} inequality, i.e., there exists a constant $c_P>0$ such that 
	\begin{align*}
		\|\bv\|_{\bL^2}\le c_P\|\nabla\bv\|_{\bL^2} \text{ for any }\bv\in \bH_0^1(\Omega).
	\end{align*}
	We explicitly write the Poincar{\'e} constant since it is necessary to keep track of the constants when proving the main result. We also recall the Ladyzhenskaya constant $c_L>0$ that satisfies
	\begin{align*}
		\|\bv\|_{\bL^4} \le \sqrt{c_L}\|\bv\|_{\bL^2}^{1/2}\|\nabla\bv\|_{\bL^2}^{1/2} \text{ for any }\bv\in\bH_0^1(\Omega).
	\end{align*}

	For the time-dependent Navier--Stokes equations, we consider functions which map $I_T=(0,T)$ to Banach spaces $X$. The space of functions from an interval $I_T$ to $X$ that can be continuously extended to the closed interval $\overline{I_T}$ will be denoted by $C(\overline{I_T};X)$. 
	The space $L^p(I_T; X)$ consists of functions $\psi:I_T\to X$ such that $t\mapsto \|\psi(t)\|_{X}$ belongs to $L^p(0,T)$, the norm in such space will be denoted as $\|\cdot\|_{L^p(X)}$. 
	Note that the elements of $L^p(I_T;L^p(\Omega))$ can be identified as elements in $L^p(\Omega_T)$.

	The spaces that we are particularly interested in are the following: 
	\begin{align*}
		& W_T^p(X):= \{v\in L^2(I_T;X): \partial_t v\in L^p(I_T;X^*) \},\\
		&\bW^{2,1}_{p,T} := \{\bpsi\in L^p(I_T;\bW^{2,p}(\Omega)\cap \bV\,): \partial_t\bpsi \in \bL^p(\Omega_T) \}.
	\end{align*}
	We see by the Aubin--Lions--Simon embedding theorem that $W_T^2(\bV\,)\hookrightarrow C(\overline{I_T}; \bH)$. 
	From Amman embedding theorem \cite[Theorem 3]{amann2001}, 
	we get the compact embedding $\bW^{2,1}_{2,T} \hookrightarrow C(\overline{I_T};\bV)$ .

	\medskip
	We conclude this section with standing assumptions. We start with the given data influencing the fluid equations.
	
	\begin{assumption}
		\label{datassum}
		The given data acting on the fluid are assumed to satisfy
		\begin{enumerate}[label=(\roman*)]
			\item homogeneous Dirichlet boundary data, i.e., $\bg = \bg_s = 0$;
			\item the given external forces satisfy $\bf\in L^2(I_T;\bL^2(\Omega))$ and $\bf_s\in \bL^2(\Omega)$.
		\end{enumerate}
	\end{assumption}
	It is well-known that uniqueness of weak solutions to the stationary Navier--Stokes equations is not guaranteed unless smallness of external force or sufficient viscosity is imposed. We shall therefore impose the following assumption for the said uniqueness.
	\begin{assumption}
		\label{statdatassum}
		The external force $\bf_s\in \bL^2(\Omega)$ and the kinematic viscosity $\mu>0$ satisfy
		\begin{align}\label{data:smallness}
			{2}c_P^2c_L\|\bf_s\|_{\bL^2} < \mu^2.
		\end{align}
	\end{assumption}

	To facilitate the convergence analysis required in our main result, we assume additional properties for the source functions.
	\begin{assumption}\label{sourceassump}
		There exists $K>0$ such that
		\begin{align*}
			\int_0^T e^{\varsigma t} \| \bf(t) - \bf_s \|_{\bL^2}^2 \du t \le K,
		\end{align*}
		for all $T>0$, where $\varsigma : = \frac{\mu^4 - 2c_L^2c_P^4\|\bf_s\|_{\bL^2}^2}{c_p^2\mu^3}  >0$.
	\end{assumption}
	A direct computation shows that this implies
	\begin{align}\label{connect:ffs}
		\|\bf\|_{L^2(\bL^2)} \le \sqrt{K} + \sqrt{T}\|\bf_s\|_{\bL^2}.
	\end{align}
	For the functions $\alpha_\varepsilon$ and $\beta_\varepsilon$, that describe the porous medium, we impose the following properties.
	\begin{assumption}
		\label{porousassum}
		The functions $\alpha_\varepsilon$ and $\alphat_\varepsilon$ satisfy
		\begin{enumerate}[label=(\roman*)]
			\item The functions $\alpha_\varepsilon:\mathbb{R}\mapsto\mathbb{R}$ and $\alphat_\varepsilon:\mathbb{R}\mapsto\mathbb{R}$ are nonnegative, satisfy $\alpha_\varepsilon,\alphat_\varepsilon\in C^{0,1}(\mathbb{R})$ with $\alpha_\varepsilon(1) = \alphat_\varepsilon(1) = 0$, $\alpha_\varepsilon(-1) = \overline{\alpha}_\varepsilon >0   $ and $\alphat_\varepsilon(-1) = \overline{\beta}_\varepsilon >0$, where $(\overline{\alpha}_\varepsilon ,\overline{\beta}_\varepsilon )\to +\infty$ as $\varepsilon\to 0$.
			\label{en:assump-alpha-i}
			
			\item As operators on $\Phi_{ad}$, $\alpha_\varepsilon$ and $\alphat_\varepsilon$ are Fr{\'e}chet differentiable in the $L^2$-norm. Specifically, we assume the existence of $\mathcal{A}_\varepsilon: \Phi_{ad}\to L^p(\Omega)$ and ${\mathcal{B}}_\varepsilon: \Phi_{ad}\to L^p(\Omega)$ for $p\ge 4$, such that 
			\begin{align*}
				\alpha_\varepsilon'(\varphi)\delta\varphi = \mathcal{A}_{\varepsilon}(\varphi)\delta\varphi \text{ and } \alphat_\varepsilon'(\varphi)\delta\varphi = {\mathcal{B}}_{\varepsilon}(\varphi)\delta\varphi
			\end{align*}
			for any $\delta\varphi \in H^1(\Omega)\cap L^\infty(\Omega)$ and such that
			\begin{align*}
				\max\{\|{\mathcal{A}}_{\varepsilon}(\varphi)\|_{L^p},\|{\mathcal{B}}_{\varepsilon}(\varphi)\|_{L^p} \} \le c_{\!\mathcal{A}}
			\end{align*}
			for any $\varphi\in \Phi_{ad}$. \label{en:assump-alpha-ii}

		\end{enumerate}	
	\end{assumption}
	
	A consequence of Assumption \ref{porousassum}\ref{en:assump-alpha-i} is that for any $\varphi\in\Phi_{ad}$ both $\alpha_\varepsilon(\varphi)$ and $\alphat_\varepsilon(\varphi)$ are in $L^\infty(\Omega)$. Indeed, for $\alpha_\varepsilon(\varphi)$, if $x\in \Omega$ is such that $|\varphi(x)|\le 1$ then
	\begin{align}\label{linfalpha}
		|\alpha_\varepsilon(\varphi(x))| = |\alpha_\varepsilon(\varphi(x)) - \alpha_\varepsilon(1)| \le c |\varphi(x) - 1|\le c(|\varphi(x)| + 1) \le 2c
	\end{align}
	where the Lipschitz constant $c>0$ is dependent neither on $x\in \Omega$ nor on $\varphi\in\Phi_{ad}$. The same reasoning can be used for $\alphat_\varepsilon(\varphi)$.
	
	Additionally, the Fr{\'e}chet differentiability of $\alpha_\varepsilon$ and $\alphat_\varepsilon$ implies their continuity on $\Phi_{ad}$ with respect to the $L^2$-norm.
	
	To analyze the optimization problem, we assume that the target profile and the double-obstacle potential satisfy the following properties.
	\begin{assumption}
		\label{objassump}
		$ $
		\begin{enumerate}[label=(\roman*)]
			\item  The target profile $\bu_d$ satisfies the regularity $\bu_d\in \bL^2(\omega)$.\label{en-assump-obj-i}
			\item The potential $\Psi: \mathbb{R}\to \overline{\mathbb{R}}$ is the double-obstacle potential defined as
			\begin{align*}
				\Psi(s) = \left\{ \begin{aligned}
					&\Psi_0(s) &&\text{if }|s|\le 1,\\
					&+\infty &&\text{if }|s|>1.
				\end{aligned} \right. \quad = \Psi_0(s) + I_{[-1,1]}(s)
			\end{align*}
			where $\Psi_0(s) = (1-s^2)/2$ and $I_{[-1,1]}(s)$ is the convex indicator function of the admissible range $[-1,1]$. \label{en-assump-obj-ii}
			The constant $c_0$ has consequently the value $c_0 = \frac{\pi}{2}$.
		\end{enumerate}
	\end{assumption}

	\begin{remark}
		We restrict our analysis to the two-dimensional Navier--Stokes equations due to the uniqueness and regularity issues for the three-dimensional (3D) case. To be precise, despite the weak-strong uniqueness results for solutions of the 3D Navier--Stokes, such result holds only locally in time.
	\end{remark}

	\section{Solutions of the governing equations}
	\label{sec:solGovEq}
	
	In this section we investigate and collect existence results of solutions to the time-dependent and the stationary Navier--Stokes equation, i.e.,  \eqref{poreq:timeNS} and \eqref{poreq:statNS}.
	
	To simplify the notations, we introduce the bilinear operator $a_\varphi:\bH_0^1(\Omega)\times \bH_0^1(\Omega) \to \mathbb{R}$ defined as
	\begin{align*}
		a_\varphi(\bu,\bv) = \mu\int_\Omega \nabla\bu:\nabla\bv\du x + \int_\Omega \alpha_\varepsilon(\varphi)\bu\cdot\bv\du x
	\end{align*}
	and the trilinear operator $b:\bL^4(\Omega)\times \bH_0^1(\Omega)\times \bL^4(\Omega)\to\mathbb{R}$, emanating from the nonlinearity in the Navier--Stokes equations, defined as 
	\begin{align*}
		b(\bu,\bv,\bw) = \int_\Omega [(\bu\cdot\nabla)\bv]\cdot \bw \du x.
	\end{align*}
	The operator $a_\varphi$ is continuous on $\bV\times\bV$ and coercive in $\bV$, i.e. there exists $c>0$ such that
	\begin{align*}
		c\|\bpsi\|_{\bV}^2 \le a_\varphi(\bpsi,\bpsi) \text{ for any }\bpsi\in\bV.
	\end{align*}
	To establish such inequality, one would have to utilize the nonnegativity of $\alpha_\varepsilon(\varphi)$.
	
	On the other hand, the operator $b$ is continuous on $\bV\times\bV\times \bV$. Furthermore, given $\bu\in\bV$ the bilinear operator $b(\bu,\cdot,\cdot):\bH_0^1(\Omega)\times \bH_0^1(\Omega)\to\mathbb{R}$ is antisymmetric, i.e.
	\begin{align*}
		b(\bu,\bv,\bw) = -b(\bu,\bw,\bv) \text{ for any }\bv,\bw\in \bH_0^1(\Omega).
	\end{align*}

	\subsection{Existence of solutions for time-dependent governing equations}
	
	We now discuss the existence of solutions to the governing equations. The proofs of the results below will be omitted as they are well-known. We refer the reader, for example, to \cite{abergel1990,temam1977} for further details.
	
	\subsubsection*{Design-to-state operator}
	Let us introduce which notion of solution we shall utilize for the time-dependent Navier--Stokes equation \eqref{poreq:timeNS}. For a given $\varphi\in \Phi_{ad}$ and $\bu_0 \in \bH$, we say that $\bu\in W_T^2(\bV)$ is a weak solution to \eqref{poreq:timeNS} if it solves the equation
	\begin{align}
		\label{weak:timeNS}
		\langle \partial_t\bu(t),\bpsi \rangle_{\bV} + a_\varphi(\bu(t),\bpsi) + b(\bu(t),\bu(t),\bpsi) = (\bf(t),\bpsi)_\Omega\quad \forall \bpsi\in \bV \text{ and a.e. }t\in(0,T),
	\end{align}
	and $\bu(0) = \bu_0$ in $\bH$.
	
	The following theorem establishes the existence and uniqueness of weak solution.
	
	\begin{theorem}\label{theorem:timeexist}
		Let Assumptions \ref{datassum} and \ref{porousassum}\ref{en:assump-alpha-i} hold, and $\bu_0\in \bH$. Then a unique weak solution $\bu\in W_T^2(\bV)$ to \eqref{poreq:timeNS} exists and satisfies
		\begin{align}
			&\|\bu\|_{L^\infty(\bH)}^2 + \mu\|\bu\|_{L^2(\bV)}^2 \le \frac{c_P^2}{\mu}\| \bf\|_{L^2(\bL^2)}^2 + \|\bu_0\|_{\bH}^2,\label{energy:time}\\
			&\|\partial_t\bu\|_{L^2(\bV)} \le c_P\left( \sqrt{T}\|\alpha_\varepsilon(\varphi)\|_{L^\infty}\|\bu\|_{L^\infty(\bH)} + \|\bf\|_{L^2(\bL^2)} \right) + \left(c_L\|\bu\|_{L^\infty(\bH)} + \mu \right)\|\bu\|_{L^2(\bV)}.\label{energy:timepr}
		\end{align}
	\end{theorem}
	
	\subsubsection*{Linearized design-to-state operator}
	The analysis on the optimization problem includes linearization of the design-to-state operator, 
	which compels us to analyze the following linear system
	\begin{align}\label{poreq:lintimeNS}
		&\left\{
		\begin{aligned}
			\partial_t \bud + \alpha_\varepsilon(\varphi)\bud - \mu\Delta \bud + (\bud\cdot\nabla)\bu_1 + (\bu_2\cdot\nabla)\bud + \nabla \pd & = \bF &&\text{ in }\Omega_T,\\
			\dive\bud & = 0 && \text{ in }\Omega_T,\\
			\bud & = 0 &&\text{ on }\Sigma_T,\\
			\bud(0) & = \bud_0 &&\text{ in }\Omega,
		\end{aligned}
		\right.
	\end{align}
	where $\varphi\in\Phi_{ad}$, $\bF\in L^2(I_T;\bL^2(\Omega))$, $\bu_1,\bu_2 \in L^4(I_T;\bL^4(\Omega))$ and $\bud_0\in \bH$.

	We define a  weak solution $\bud\in W_T^2(\bV)$  as  solving the variational equation
	\begin{align}
		\label{weak:lintimeNS}
		\langle \partial_t\bud(t),\bpsi \rangle_{\bV} + a_\varphi(\bud(t),\bpsi) + b(\bud(t),\bu_1(t),\bpsi) + b(\bu_2(t),\bud(t),\bpsi) = (\bF(t),\bpsi)_\Omega\quad \forall \bpsi\in \bV
	\end{align}
	for a.e. $t\in(0,T)$ and $\bud(0) = 0$ in $\bH$.
	
	\begin{theorem}
		\label{theorem:lintimeexist}
		Let Assumption \ref{porousassum}\ref{en:assump-alpha-i} hold, $\bF\in L^2(I_T;\bL^2(\Omega))$ and $\bu_1,\bu_2\in L^4(I_T;\bL^4(\Omega))$. Then a unique weak solution to \eqref{poreq:lintimeNS} exists and satisfies
		\begin{align}
			&\begin{aligned}
				\|\bud\|_{L^\infty(\bH)}^2 + \frac{\mu}{2}\|\bud\|_{L^2(\bV)}^2  \le \|\bud_0\|_{\bH}^2 + \frac{c_p^2}{\mu}\| \bF\|_{L^2(\bL^2)}^2,\label{energy:lintime}
			\end{aligned}\\
			\nonumber\\
			&\begin{aligned}
				\|\partial_t\bud\|_{L^2(\bV^*)}  &\le c_P\|\bF\|_{L^2(\bL^2)} + (c_P^2\|\alpha_\varepsilon(\varphi)\|_{L^\infty} + \mu)\|\bud\|_{L^2(\bV)} \\
				&+ \sqrt{c_L}( \|\bu_1 \|_{L^4(\bL^4)} + \|\bu_2 \|_{L^4(\bL^4)} )\|\bud \|_{L^\infty(\bH)}^{1/2}\|\bud \|_{L^2(\bV)}^{1/2}.
			\end{aligned} 
			\label{energy:lintimepr}
		\end{align}
		
	\end{theorem}

	\begin{remark}
		We note that to arrive at \eqref{energy:lintimepr}, we used the identity
		\begin{align*}
			\|\bud\|_{L^4(\bL^4)} \le \sqrt{c_L}\|\bud \|_{L^\infty(\bH)}^{1/2}\|\bud \|_{L^2(\bV)}^{1/2}.
		\end{align*}
		This implies that Theorem~\ref{theorem:lintimeexist} still holds when we impose the assumption $\bu_1,\bu_2\in W_T^2(\bV)$.
	\end{remark}
	\medskip
	
	Aside from the energy estimates \eqref{energy:lintime} and \eqref{energy:lintimepr},  the weak solution to \eqref{poreq:lintimeNS} satisfies the following estimate.
	\begin{proposition}\label{prop:decaylin}
		Suppose that the assumptions in Theorem~\ref{theorem:lintimeexist} hold. Assume further that there exists $\bz\in L^2(I_T;\bL^2(\Omega))$ such that $\|\nabla\bu_1(t)\|_{\bL^2}^2 \le \|\bz(t)\|_{\bL^2}^2$ for almost every $t\in[0,T]$. Then the weak solution $\bud\in W^2_T(\bV)$ of \eqref{poreq:lintimeNS} satisfies
		\begin{align}\label{est:decay}
			\|\bud(t)\|_{\bH}^2 \le e^{A(t)}\|\bud_0\|_{\bH}^2 + \frac{2c_P^2}{\mu}\int_0^te^{A(t)-A(\tau)}\|\bF(\tau)\|_{\bL^2}^2 \du \tau,
		\end{align}
		where $A(t) = \frac{2c_L^2}{\mu}\int_0^t\|\bz(\tau)\|_{\bL^2}^2\du \tau - \frac{\mu}{c_P^2}t$.
	\end{proposition}
	\begin{proof}
		Since
		\begin{align*}
			\frac{1}{2}\frac{d}{dt}\|\bud\|_{\bH}^2 + \frac{\mu}{2}\|\bud\|_{\bV}^2 \le \frac{c_P^2}{\mu}\|\bF\|_{\bL^2}^2 + \frac{c_L^2}{\mu}\|\nabla\bu_1\|_{\bL^2}^2\|\bud\|_{\bH}^2
		\end{align*}
		Poincare inequality gives us
		\begin{align*}
			\frac{d}{dt}\|\bud\|_{\bH}^2  \le \frac{2c_P^2}{\mu}\|\bF\|_{\bL^2}^2 + 2\left( \frac{c_L^2}{\mu}\|\nabla\bu_1\|_{\bL^2}^2 - \frac{\mu}{2c_P} \right)\|\bud\|_{\bH}^2.
		\end{align*}
		A direct application of Gronwall lemma gives us \eqref{est:decay}.
	\end{proof}

	\subsubsection*{Adjoint design-to-state operator}
	To facilitate the analysis of the first-order optimality conditions, we will utilize the adjoint of the linearization of the design-to-state operator. This motivates us to consider the following adjoint system 
	\begin{align}\label{poreq:adjtimeNS}
		&\left\{
		\begin{aligned}
			-\partial_t \bua + \alpha_\varepsilon(\varphi)\bua - \mu\Delta \bua + (\nabla\bu_1)^{\top}\bua - (\bu_2\cdot\nabla)\bua + \nabla \pa & = \bG &&\text{ in }\Omega_T,\\
			\dive\bua & = 0 && \text{ in }\Omega_T,\\
			\bua & = 0 &&\text{ on }\Sigma_T,\\
			\bua(T) & = \bua_T &&\text{ in }\Omega,
		\end{aligned}
		\right.
	\end{align}
	where $\bG\in L^2(I_T;\bL^2(\Omega))$, $\bu_1,\bu_2 \in W^2_T(\bV)$ and $\bua_T\in \bH$.
	
	A weak solution  $\bua\in W_T^{4/3}(\bV)$  to the adjoint system \eqref{poreq:adjtimeNS} satisfies the variational equation
	\begin{align}
		\label{weak:adjtimeNS}
		-\langle \partial_t\bua(t),\bpsi \rangle_{\bV} + a_\varphi(\bua(t),\bpsi) + b(\bpsi,\bu(t),\bua(t)) + b(\bu(t),\bpsi,\bua(t)) = (\bG(t),\bpsi)_\Omega\quad \forall \bpsi\in \bV
	\end{align}
	for a.e. $t\in(0,T)$ and $\bua(T) = 0$ in $\bH$.

    \jshsr{Analogously to \cite[Lemma 1.118]{hinze2009} one infers that \eqref{poreq:adjtimeNS} admits a unique weak solution. For the subsequent analysis it is convenient to have the following uniqueness and regularity result for the adjoint variable $\bua$.}
	\begin{theorem}
		\label{theorem:adjtimeexist}
		Let Assumption \ref{porousassum}\ref{en:assump-alpha-i} hold, $\bG\in L^2(I_T;\bL^2(\Omega))$, and $\bu_1,\bu_2\in W^2_T(\bV)$. Then\jshsr{, \eqref{poreq:adjtimeNS} admits a unique weak solution which satisfies}
		\begin{align} 
			&\begin{aligned}
				\|\bua\|_{L^\infty(\bH)}^2 + \mu\|\bua\|_{L^2(\bV)}^2 \le \frac{c_P^2}{\mu}\| \bG\|_{L^2(\bL^2)}^2 + \|\bua_T\|_{\bH}^2,
			\end{aligned}\label{energy:adjtime}\\
			\nonumber \\
			&\begin{aligned}
				\|\partial_t\bua\|_{L^{4/3}(\bV^*)}^{4/3} \le c_P\| \bG\|_{L^2(\bL^2)}T^{\frac{1}{4}} + \mu \|\bua\|_{L^2(\bV)}T^{\frac{1}{4}} + c_P\|\alpha_\varepsilon(\varphi)\|_{L^\infty}\|\bua\|_{L^\infty(\bH)}\\
				+c_Lc_p\left( \|\bu_1\|_{L^2(\bV)}^{\frac{3}{2}} + \|\bu_2\|_{L^2(\bV)}^{\frac{3}{2}} + \|\bua\|_{L^\infty(\bH)}^{\frac{3}{2}}\|\bua\|_{L^2(\bV)}^{\frac{3}{2}} \right).
			\end{aligned}\label{energy:adjtimepr}
		\end{align}
	\end{theorem}

	\medskip
	\subsubsection*{Higher regularity of the solutions}
	We note that due to the assumptions on $\bf$, $\bF$, and $\bG$, we can readily recover higher regularity for the weak solutions of \eqref{weak:timeNS}, \eqref{weak:lintimeNS} and \eqref{weak:adjtimeNS} provided that we impose higher regularity for the auxiliary variables $\bu_1$ and $\bu_2$, the initial data $\bu_0$, $\bud_0$ and the terminal data $\bua_T$. 
	
	\begin{proposition}\label{prop:timestrong}
		Suppose that Assumption \ref{porousassum}\ref{en:assump-alpha-i} hold, $\bf,\bF,\bG\in L^2(I_T;\bL^2(\Omega))$, $\bu_1,\bu_2\in \bW^{2,1}_{2,T}$ and $\bu_0,\bud_0,\bua_T\in \bV$. 
		Then the weak solutions of \eqref{poreq:timeNS}, \eqref{poreq:lintimeNS}, and of \eqref{poreq:adjtimeNS} satisfy $\bu,\bud,\bua\in \bW^{2,1}_{2,T}$.
	\end{proposition}

	\begin{proof}[Proof (Sketch).]
		The proof of Proposition~\ref{prop:timestrong} is not new, and the arguments, for example, in \cite{jork2025,gerhardt1978} can be followed accordingly. The crucial part lies in obtaining the estimate for the following integral, which arises from testing the first equation in \eqref{poreq:timeNS} with $\Delta \bu$
		\begin{align*}
			\left| \int_\Omega \alpha_\varepsilon'(\varphi) (\nabla\varphi\otimes \bu ): \nabla \bu \du x \right| \le C \|\nabla\varphi\|_{\bL^2} \|\bu\|_{\bL^{p_1}} \|\nabla\bu\|_{\bL^{p_2}} \le C \|\nabla\varphi\|_{\bL^2} \|\bu\|_{\bL^2} \|\Delta\bu\|_{\bL^2},
		\end{align*}
		where $\frac{1}{2}=\frac{1}{p_1}+\frac{1}{p_2}$.
		We note that the last inequality follows from the Gagliardo–Nirenberg interpolation inequality, namely
		\begin{align*}
			\|\bu\|_{\bL^{p_1}} \le C\|\bu\|_{\bL^2}^{1/2 + 1/p_1}\|\Delta\bu\|_{\bL^2}^{1/2 - 1/p_1}
			\quad\text{and}\quad \|\nabla\bu \|_{\bL^{p_2}} \le C\|\bu\|_{\bL^2}^{1/p_2}\|\Delta\bu\|_{\bL^2}^{1 - 1/p_2}.
		\end{align*}
	\end{proof}

	

		

	\subsection{Existence of solutions for the stationary governing equations}
	We next discuss the existence of solutions to the stationary equations.
	As in the case of the time-dependent system, we are also interested in the linear case as well as an adjoint system for the stationary problem. For proofs of the corresponding results we again refer to, e.g., \cite{abergel1990,temam1977}.

	\subsubsection*{Design-to-state operator}
	Given a design function $\varphi\in \Phi_{ad}$, we say that an element $\bv\in \bV$ is a weak solution to the stationary Navier--Stokes equation \eqref{poreq:statNS} if it satisfies the variational equation
	\begin{align}\label{variation:stat}
		a_\varphi(\bv,\bpsi) + b(\bv,\bv,\bpsi) = (\bf_s,\bpsi)_\Omega\qquad \forall\bpsi\in \bV.
	\end{align}
	
	\begin{theorem}\label{theorem:statexist}
		Suppose that Assumptions~\ref{datassum} and \ref{porousassum}\ref{en:assump-alpha-i} hold. Then for a given design function $\varphi\in \Phi_{ad}$, there exists a weak solution to \eqref{poreq:statNS} satisfying
		\begin{align}
			\label{energy:stat}
			\|\bv\|_{\bV} \le \frac{c_P}{\mu}\|\bf_s\|_{\bL^2}.
		\end{align}
		\jshsr{If we further assume that the external force $\bf_s$ and the kinematic viscosity $\mu>0$ satisfy Assumption \ref{statdatassum}	
		then the weak solution to \eqref{poreq:statNS} is unique.}
	\end{theorem}

	From here onwards, we work on the condition that Assumption~\ref{statdatassum} always holds. So that whenever we talk about a weak solution of the stationary equation \eqref{poreq:statNS}, we know that it is unique.
	Furthermore, we see from \eqref{energy:stat} and \eqref{data:smallness} that any weak solution $\bv\in \bV$ of \eqref{poreq:statNS} satisfies
	\begin{align}\label{energy:statsmall}
		\|\bv\|_{\bV} < \frac{\mu}{{2}c_Pc_L}.
	\end{align}
	It is thus natural to talk about elements of $\bV$ that satisfy \eqref{energy:statsmall}, as we shall see later.
	
	\subsubsection*{Linearized design-to-state operator}
	For $\bv_1,\bv_2 \in \bV$ and $\bF_s \in \bL^2(\Omega)$ we consider the linearized system
	\begin{align}
		\label{poreq:linstatNS}
		\left\{
		\begin{aligned}
			\alpha_\varepsilon(\varphi)\bvd - \mu\Delta \bvd + (\bvd\cdot\nabla)\bv_1 + (\bv_2\cdot\nabla)\bvd + \nabla \pid & = \bF_s &&\text{ in }\Omega,\\
			\dive\bvd & = 0 && \text{ in }\Omega,\\
			\bvd & = 0 &&\text{ on }\Sigma.
		\end{aligned}
		\right.
	\end{align}
	We define a weak solution $\bvd\in \bV$ to \eqref{poreq:linstatNS} if it solves the following variational problem
	\begin{align}\label{variation:linstat}
		a_\varphi(\bvd,\bpsi) + b(\bvd,\bv_1,\bpsi) + b(\bv_2,\bvd,\bpsi) = (\bF_s,\bpsi)_\Omega\qquad \forall\bpsi\in \bV.
	\end{align}
	The existence of solution to \eqref{variation:linstat} and its uniqueness is the subject of Theorem~\ref{theorem:linstatexist}.
	\begin{theorem}\label{theorem:linstatexist}
		Suppose that Assumption \ref{porousassum}\ref{en:assump-alpha-i} holds, $\bF_s\in L^2(\Omega)$ and that $\bv_1,\bv_2\in \bV$ with $\bv_1$ satisfying \eqref{energy:statsmall}. Then, given $\varphi\in \Phi_{ad}$, there exists a unique weak solution $\bvd\in\bV$ to \eqref{poreq:linstatNS} satisfying 
		\begin{align*}
			\|\bvd\|_{\bV} \le \frac{c_P}{\sqrt{\mu \sigma}}\|\bF_s\|_{\bL^2},
		\end{align*}
		where $\sigma = \mu - 2 c_Lc_P\|\bv_1\|_{\bV} > 0$.
	\end{theorem}
	As one would notice, the auxiliary variables $\bv_1,\bv_2\in \bV$ will be weak solutions of the stationary Navier--Stokes equations \eqref{poreq:statNS}, hence the imposition of the estimate \eqref{energy:statsmall}. Although the proof of Theorem~\ref{theorem:linstatexist} is straightforward, let us explain the imposition of \eqref{energy:statsmall} on $\bv_1\in\bV$. In fact, the reason for assuming such property is two-fold:
	\begin{itemize}
		\item coercivity of the left-hand side with respect to $\bV$;
		\item uniqueness of the weak solution.
	\end{itemize}
	For the coercivity, we additionally use the positivity of $\alpha_\varepsilon(\varphi)$, H{\"o}lder and Poincare inequalities and the anti-symmetry of $b(\bv_2,\cdot,\cdot)$
	\begin{align*}
		a_\varphi(\bvd,\bvd) + b(\bvd,\bv_1,\bvd) + b(\bv_2,\bvd,\bvd) & \ge \mu\|\nabla\bvd\|_{L^2}^2 - c_Lc_P\|\nabla\bv_1\|_{L^2}\|\nabla\bvd\|_{L^2}^2\\
		& = \left( \mu - c_Lc_P\|\nabla\bv_1\|_{L^2} \right)\|\bvd\|_{\bV}^2.
	\end{align*}
	From \eqref{energy:statsmall}, we know that $\mu > 2c_Lc_P\|\nabla\bv_1\|_{L^2} > c_Lc_P\|\nabla\bv_1\|_{L^2}$, and thus the term multiplied to $\|\nabla\bvd\|_{L^2}^2$ is positive. Similar line of reasoning can be used to show that indeed \eqref{energy:statsmall} gives us the uniqueness of the weak solution.
	
	\subsubsection*{Adjoint design-to-state operator}
	For $\bv_1,\bv_2 \in \bV$ and $\bG_s \in \bL^2(\Omega)$ we consider the  adjoint system
	\begin{align}
		\label{poreq:adjstatNS}
		\left\{
		\begin{aligned}
			\alpha_\varepsilon(\varphi)\bva - \mu\Delta \bva + (\nabla\bv_1)^\top \bva - (\bv_2\cdot\nabla)\bva + \nabla \pia & = \bG_s &&\text{ in }\Omega,\\
			\dive\bva & = 0 && \text{ in }\Omega,\\
			\bva & = 0 &&\text{ on }\Sigma.
		\end{aligned}
		\right.
	\end{align}
	A variable $\bva\in\bV$ is a weak solution to \eqref{poreq:adjstatNS} if it solves
	\begin{align}\label{variation:adjstat}
		a_\varphi(\bva,\bpsi) + b(\bpsi,\bv_1,\bva) + b(\bv_2,\bpsi,\bva) = (\bG_s,\bpsi)_\Omega\qquad \forall\bpsi\in \bV.
	\end{align}
	
	The existence of solutions for the adjoint equations \eqref{poreq:adjstatNS} is not as nuanced as the time-dependent case. It is in fact as straightforward as the existence of weak solutions for the linear equation \eqref{poreq:linstatNS}.
	
	\begin{theorem}\label{theorem:adjstatexist}
		Suppose that Assumption \ref{porousassum}\ref{en:assump-alpha-i} holds, $\bG_s\in \bL^2(\Omega)$ and that $\bv_1,\bv_2\in \bV$ with $\bv_1$ satisfying \eqref{energy:statsmall}. Then, given $\varphi\in \Phi_{ad}$, there exists a unique weak solution $\bva\in\bV$ to \eqref{poreq:linstatNS} satisfying 
		\begin{align*}
			\|\bva\|_{\bV} \le \frac{c_P}{\sqrt{\mu \sigma}}\|\bG_s\|_{\bL^2},
		\end{align*}
		where $\sigma = \mu - 2 c_Lc_P\|\bv_1\|_{\bV} > 0$.
	\end{theorem}
	
	\subsubsection*{Higher regularity of the solutions}
	Using classical elliptic regularity, improvement on the auxiliary variables $\bv_1$ and $\bv_2$ improves the regularity of solutions to \eqref{poreq:linstatNS} and \eqref{poreq:adjstatNS}, while a simple fixed point argument gives us the regularity of the weak solution to \eqref{poreq:statNS}. Again, we skip the proof but we refer the reader to \cite[Lemma~4.3]{GHHKL2016-ifb}.
	\begin{proposition}\label{prop:statstrong}
		Suppose that Assumption \ref{porousassum} hold, $\bf_s,\bF_s,\bG_s\in\bL^2(\Omega)$ and $\bu_1,\bu_2\in\bV\cap \bH^2(\Omega)$. Then the weak solutions $\bv$ of \eqref{poreq:statNS}, $\bvd$ of \eqref{poreq:linstatNS}, and $\bva$ of \eqref{poreq:adjstatNS} satisfy $\bv,\bvd,\bva\in\bV\cap \bH^2(\Omega)$.
	\end{proposition}

	\section{The optimization problems}
	\label{sec:optProblems}
	Before we recall the optimization problems, let us note that the existence and uniqueness results from Section~\ref{sec:solGovEq} allow us to define the following operators, which we respectively call the time-dependent and stationary design-to-state operators
	\begin{itemize}
		\item $\ST:\Phi_{ad}\to W_T^2(\bV)$ defined as the unique weak solution $\ST(\varphi) = \bu$  of \eqref{poreq:timeNS};
		\item $\Ss:\Phi_{ad}\to \bV$ defined as the unique weak solution $\Ss(\varphi) = \bv$ of \eqref{poreq:statNS}.
	\end{itemize}
	The well-definedness of both operators are attributed to Theorem~\ref{theorem:timeexist} and Theorem~\ref{theorem:statexist}, respectively. 
	
	From these, we consider the reduced forms of the functionals $J_T$ and $J_s$, and thus the  optimization problems under consideration
	\begin{align}
		&\min_{\varphi\in\Phi_{ad}} J_T(\varphi) = \frac{1}{T}\left[\int_{\Omega_T} \left( \frac{\varphi + 1}{2} \right) |\ST(\varphi)-\bu_d|^2 \one_\omega \du x\du t + \int_{\Omega_T} \alphat_\varepsilon(\varphi)|\ST(\varphi)|^2\du x\du t \right]  + \gamma \mathbb{E}_{\varepsilon}(\varphi)
		\tag{\ensuremath{P_T}}\label{reduced:objphasetime}\\
		&\min_{\varphi\in\Phi_{ad}} J_s(\varphi) = \int_\Omega \left( \frac{\varphi + 1}{2} \right) |\Ss(\varphi)-\bu_d|^2\one_\omega \du x  + \int_\Omega \alphat_\varepsilon(\varphi)|\Ss(\varphi)|^2\du x + \gamma \mathbb{E}_{\varepsilon}(\varphi)
		\tag{\ensuremath{P_s}}\label{reduced:objphasestat}
	\end{align}
	To talk about the existence of minimizers to \ref{reduced:objphasetime} and \ref{reduced:objphasestat}, let us discuss the topology we endow $\Phi_{ad}$ with. On the one hand, the definition of $\Phi_{ad}$ compels us to consider the weak$^*$ topology in $L^\infty(\Omega)$. Additionally it is necessary to also consider the weak topology in $H^1(\Omega)$. The following definition formalizes such a topology for elements in $\Phi_{ad}$.
	
	\begin{definition}
		Let $(\varphi_n)\subset \Phi_{ad}$ and $\varphi\in\Phi_{ad}$. We say that the sequence $(\varphi_n)$ converges to $\varphi$ in $\Phi_{ad}$, which we write symbolically as $\varphi_n\wPhi \varphi$, if and only if $\varphi_n \ws \varphi$ in $L^\infty(\Omega)$ and $\varphi_n \rightharpoonup \varphi$ in $H^1(\Omega)$.
	\end{definition}
	
	We note that a sequence $(\varphi_n)\subset\Phi_{ad}$ converging to $\varphi\in \Phi_{ad}$ in $\Phi_{ad}$ satisfies $\varphi_n\to\varphi$ in $L^p(\Omega)$, up to a subsequence, for $2\le p < \infty$ and a.e. in $\Omega$.
	
	The next question that we  address is the continuity of the design-to-state operators. Let us first deal with the time-dependent design-to-state operator $\ST$.
	\begin{theorem}[Continuity of $\ST$]
		\label{thm:ST-continuous}
		Let $(\varphi_n)\subset \Phi_{ad}$ and $\varphi^*\in\Phi_{ad}$ such that $\varphi_n \wPhi \varphi^*$. Suppose further that Assumptions \ref{datassum}, \ref{porousassum}\ref{en:assump-alpha-i} and \ref{objassump}\ref{en-assump-obj-i} hold.
		Then we obtain the following convergence results
		\begin{align}
			\bullet \ & \ST(\varphi_n)\to \ST(\varphi^*) \text{ in } L^2(I_T;\bV), \text{ i.e., continuity of } \ST,\label{eq:ST-continuous}\\
			\bullet \ &\int_{\Omega_T} \left( \frac{\varphi_n + 1}{2} \right) |\ST(\varphi_n)-\bu_d|^2 \one_\omega \du x\du t  \to \int_{\Omega_T} \left( \frac{\varphi^* + 1}{2} \right) |\ST(\varphi^*)-\bu_d|^2 \one_\omega \du x\du t, \label{conv:timetracking}\\
			\bullet \ & \int_{\Omega_T} \alphat_\varepsilon(\varphi_n)|\ST(\varphi_n)|^2\du x\du t \to \int_{\Omega_T} \alphat_\varepsilon(\varphi^*)|\ST(\varphi^*)|^2\du x\du t.\label{conv:timehecht}
		\end{align}
	\end{theorem}
	\begin{proof} 
		To achieve \eqref{eq:ST-continuous} we start with diagonal testing to obtain weakly convergent subsequences.
		Since $\bu_n:= \ST(\varphi_n)$ solves \eqref{weak:timeNS} with $\varphi = \varphi_n$, a diagonal testing gives us  uniform boundedness of the sequence $(\bu_n)$ in $L^\infty(I_T;\bH)\cap L^2(I_T;\bV)$ and in $H^1(I_T;\bV^*)$. From this, we infer the existence of $\bu^*\in W^2_T(\bV)$ such that
		\begin{align*}
			\bu_n \rightharpoonup \bu^* &\text{ in }L^2(I_T;\bV),\\
			\bu_n \ws \bu^* &\text{ in }L^\infty(I_T;\bH),\\
			\partial_t \bu_n \rightharpoonup \partial_t \bu^* &\text{ in }L^2(I_T;\bV^*).
		\end{align*}
		Using the convergences we have at hand, a passage to the limit gives us $\bu^* = \ST(\varphi^*)$. Furthermore, we get, from Aubin--Lions lemma, that 
		\begin{align}\label{est1}
			\bu_n\to\bu^* \text{ in }L^2(I_T;\bH).
		\end{align}
		Now, the variable $\bw := \bu_n-\bu^* \in W_T^2(\bV)$ solves the equation
		\begin{align*}
			\langle \partial_t\bw, \bpsi \rangle_{\bV} + a_{\varphi_n}(\bw,\bpsi) + b(\bu_n,\bw,\bpsi) = b(\bw,\bpsi,\bu^*) + ( (\alpha_\varepsilon(\varphi^*) - \alpha_\varepsilon(\varphi_n) )\bu^*,\bpsi )_\Omega\quad \forall\bpsi\in \bV.
		\end{align*}
		Taking $\bpsi = \bw$ --- we get from the nonnegativity of $\alpha_\varepsilon(\varphi_n)$ --- that 
		\begin{align*}
			\frac{1}{2}\frac{d}{dt}\|\bw\|_{\bL^2}^2 + \mu\|\nabla \bw\|_{\bL^2}^2 \le c_L\|\nabla\bu^*\|_{\bL^2}\|\bw\|_{\bL^2}\|\nabla\bw\|_{\bL^2} + c_Lc_P\|\alpha_\varepsilon(\varphi^*) - \alpha_\varepsilon(\varphi_n)  \|_{L^2}\|\nabla\bu^*\|_{\bL^2}\|\nabla\bw\|_{\bL^2} \\
			\le \frac{c_L^2}{\mu}\|\nabla\bu^*\|_{\bL^2}^2\|\bw\|_{\bL^2}^2 + \frac{c_L^2c_P^2}{\mu}\|\alpha_\varepsilon(\varphi^*) - \alpha_\varepsilon(\varphi_n)  \|_{L^2}^2\|\nabla\bu^*\|_{\bL^2}^2 + \frac{\mu}{2}\|\nabla\bw\|_{\bL^2}^2. 
		\end{align*}
		\jshsb{We recall that by definition, $\bu^* = \ST(\varphi^*)$ implies $\bu^*(0) = \bu_0$.} Therefore, after rearrangement and taking integral over the interval $(0,T)$ we get
		\begin{align*}
			\frac{\mu}{2}\int_0^T \|\nabla\bw\|_{\bV}^2 \du t \le  \frac{c_L^2}{\mu}\|\nabla\bu^*\|_{L^\infty(\bL^2)}^2\int_0^T\|\bw\|_{\bL^2}^2\du t + \frac{c_L^2c_P^2}{\mu}\|\alpha_\varepsilon(\varphi^*) - \alpha_\varepsilon(\varphi_n)  \|_{L^2}^2\int_0^T\|\nabla\bu^*\|_{\bL^2}^2\du t.
		\end{align*}
		\jshsb{We also underline the fact that, from Proposition~\ref{prop:timestrong}, $\bu^* = \ST(\varphi^*)$ implies $\bu^* \in W^{2,1}_{2,T} \hookrightarrow C(\overline{I_T};\bV)$. This allowed us to get the estimate for the first term in the left-hand side of the inequality above.}  From \eqref{est1} we see that $\bw\to0$ in $L^2(I_T;\bH)$. Adding the fact that $\alpha_\varepsilon(\varphi_n)\to \alpha_\varepsilon(\varphi)$ in $L^2(\Omega)$, we can conclude that the right-hand side of the inequality goes to zero as $n\to\infty$; i.e. the stated continuity.

		To establish \eqref{conv:timetracking}, we note that
		\begin{align*}
			&\left|\int_{\Omega_T} \left( \frac{\varphi_n + 1}{2} \right) |\bu_n-\bu_d|^2 \one_\omega \du x\du t - \int_{\Omega_T} \left( \frac{\varphi^* + 1}{2} \right) |\bu^*-\bu_d|^2 \one_\omega \du x\du t  \right|\\
			&\le \left|\int_{\Omega_T} \left( \frac{\varphi_n + 1}{2} \right) (\bu_n - \bu^*)\cdot(\bu_n + \bu^* - 2\bu_d) \one_\omega \du x\du t  \right|\\
			&\quad + \frac{1}{2}\left|\int_{\Omega} (\varphi_n-\varphi^*) \int_0^T |\bu^*-\bu_d|^2\one_\omega  \du t\du x \right|\\
			&\le (\|\varphi_n\|_{L^\infty} + 1)\|\bu_n - \bu^*\|_{L^2(\bH)}
			( c_P(\|\bu_n\|_{L^2(\bV)}+ \| \bu^*\|_{L^2(\bV)}) + \|\bu_d\|_{\bL^2(\omega)} )\\
			&\quad + \frac{1}{2}\left|\int_{\Omega} (\varphi_n-\varphi^*) \int_0^T |\bu^*-\bu_d|^2 \one_\omega \du t\du x \right| = : I_1 + I_2.
		\end{align*}
		Since $\|\varphi_n\|_{L^\infty} \le 1$, $\bu_n$ and $\bu^*$ satisfy the estimate \eqref{energy:time}, and $\bu_n\to \bu^*$ in $L^2(I_T;\bH)$, we see that $I_1$ tends to zero. Finally, from the fact that $\int_0^T |\bu^*-\bu_d|^2  \du t \in L^1(\Omega)$, we infer that $I_2\to 0$ as $n\to\infty$.
		
		Similar arguments can be used to establish \eqref{conv:timehecht}.        
	\end{proof}
	
	\begin{remark}
		We note that the continuity highlighted by Theorem~\ref{thm:ST-continuous} is the strong convergence in $L^2(I_T;\bV)$, which is actually quite generous for our purposes. 
		As one would have noticed in the proof of Theorem~\ref{thm:ST-continuous}, we also get the strong convergence in $L^2(I_T;\bH)$ by invoking Aubin--Lions lemma, which is enough for our upcoming results.
	\end{remark}

	The case of the stationary design-to-state operator $\Ss$ is, as expected, less involved than the previous case. Such continuity property is stated in Theorem~\ref{theorem:statcontinuity} but we skip the proof.
	\begin{theorem}[Continuity of $\mathcal S_s$]
		\label{theorem:statcontinuity}
		Let $(\varphi_n)\subset \Phi_{ad}$ and $\varphi^*\in\Phi_{ad}$ such that $\varphi_n \wPhi \varphi^*$. Furthermore, suppose that Assumptions  \ref{datassum}, \ref{porousassum}\ref{en:assump-alpha-i} and \ref{objassump}\ref{en-assump-obj-i} hold.
		Then, similar as the time-dependent case, we obtain the following convergence
		results
		\begin{align*}
			\bullet \ & \Ss(\varphi_n)\to \Ss(\varphi^*) \text{ in } \bV.\\
			\bullet \ &\int_{\Omega} \left( \frac{\varphi_n + 1}{2} \right) |\Ss(\varphi_n)-\bu_d|^2 \one_\omega \du x  \to \int_{\Omega} \left( \frac{\varphi^* + 1}{2} \right) |\Ss(\varphi^*)-\bu_d|^2 \one_\omega \du x, \\
			\bullet \ & \int_{\Omega} \alphat_\varepsilon(\varphi_n)|\Ss(\varphi_n)|^2\du x \to \int_{\Omega} \alphat_\varepsilon(\varphi^*)|\Ss(\varphi^*)|^2\du x. 
		\end{align*}
	\end{theorem}

	\subsection{Existence of optimal design functions}
	
	After defining  the appropriate topology for the admissible designs and providing continuity for the design-to-state operators, we are now in position to show that minimizers for both $J_T$ and $J_s$ exist.
	
	\begin{theorem}\label{thm:minexist}
		Suppose that Assumptions \ref{datassum}, \ref{porousassum}\ref{en:assump-alpha-i}, \ref{objassump}\ref{en-assump-obj-i}, and \ref{objassump}\ref{en-assump-obj-ii} hold.
		Then problems \eqref{reduced:objphasetime} and \eqref{reduced:objphasestat} both admit at least one global solution, i.e., there exist $\varphi^T,\varphi^s\in\Phi_{ad}$ such that $J_T(\varphi^T)\le J_T(\varphi)$ and $J_s(\varphi^s)\le J_s(\varphi)$ for all $\varphi \in\Phi_{ad}$.
	\end{theorem}
	\begin{proof}
		We start with the case of the time-dependent state. We begin with the fact that $J_T$ is nonnegative, from which we infer the existence of an infimizing sequence $(\varphi_n)\subset \Phi_{ad}$, i.e., 
		\begin{align*}
			J^*:=\liminf_{n\to \infty} J_T(\varphi_n) = \inf_{\varphi\in\Phi_{ad}} J_T(\varphi).
		\end{align*}
		The non-negativity of $\Psi$ in $\Phi_{ad}$ and the other integrands of the objective functional allows us to get the following estimate for the gradient for any $\varphi \in \Phi_{ad}$
		\begin{align*}
			\frac{\gamma}{2c_0} \int_\Omega \frac{\varepsilon}{2}|\nabla\varphi|^2\du x \le J_T(\varphi).
		\end{align*}
		Given $\eta > 0$, one finds $N\in\mathbb{N}$ such that
		\begin{align*}
			\frac{\gamma}{2c_0} \int_\Omega \frac{\varepsilon}{2}|\nabla\varphi_n|^2\du x \le J_T(\varphi_n) \le \inf_{\varphi\in \Phi_{ad}} J_T(\varphi) + \eta
		\end{align*}
		whenever $n\ge N$. We thus infer the uniform boundedness of $(\nabla\varphi_n)$ in $\bL^2(\Omega)$. Meanwhile, we see that, by definition, both the $L^\infty$ norm and the average $\frac{1}{|\Omega|}\int \varphi \du x$ of any $\varphi\in\Phi_{ad}$ is bounded by 1. Hence, we get the uniform boundedness of the sequence $(\varphi_n)$ both in $H^1(\Omega)$ and $L^\infty(\Omega)$. We can thus find an element $\varphi^T\in \Phi_{ad}$ such that $\varphi_n\wPhi \varphi^T$ in $\Phi_{ad}$, up to a subsequence. From \eqref{conv:timetracking} and \eqref{conv:timehecht} we see that
		\begin{align*}
			\int_{\Omega_T} \left( \frac{\varphi_n + 1}{2} \right)& |\ST(\varphi_n)-\bu_d|^2 \one_\omega \du x\du t + \int_{\Omega_T} \alphat_\varepsilon(\varphi_n)|\ST(\varphi_n)|^2\du x\du t\\
			& \longrightarrow \int_{\Omega_T} \left( \frac{\varphi^T + 1}{2} \right) |\ST(\varphi^T)-\bu_d|^2 \one_\omega \du x\du t + \int_{\Omega_T} \alphat_\varepsilon(\varphi^T)|\ST(\varphi^T)|^2\du x\du t.
		\end{align*}
		While the weak lower-semicontinuity of the $L^2$-norm, the weak convergence $\nabla\varphi_n\rightharpoonup\nabla\varphi^T$ in $\bL^2(\Omega)$, and $\varphi_n\to\varphi^T$ in $L^p(\Omega)$ gives us
		\begin{align*}
			\mathbb{E}_\varepsilon(\varphi^T)\le \liminf_{n\to\infty} \mathbb{E}_\varepsilon(\varphi_n).
		\end{align*}
		Combining these, we get
		\begin{align*}
			J_T(\varphi^T)\le \liminf_{n\to\infty} J_T(\varphi_n) = J^*.
		\end{align*}
		These imply that $\varphi^T$ is a minimizer for $J_T$.
		
		Similar arguments can be used to establish the existence of a minimizer for $J_s$.
	\end{proof}

	\subsection{First-order optimality conditions}
	
	Let us consider the operator $\mathcal G: W_T^2(\bV)\times \Phi_{ad} \to L^2(I_T;\bV^*)\times \bH$ defined as $\mathcal{G}(\bu,\varphi) = (\mathcal{G}_1(\bu,\varphi),\mathcal{G}_2(\bu,\varphi))$ with 
	\begin{align*}
		\int_0^T \langle \mathcal{G}_1(\bu,\varphi),\bpsi \rangle_{\bV} \du t
		= \int_0^T \left[ \langle \partial_t\bu,\bpsi \rangle_{\bV} + a_\varphi(\bu,\bpsi) + b(\bu,\bu,\bpsi)  - (\bf,\bpsi)_\Omega \right] \du t.
	\end{align*}
	and $\mathcal{G}_2(\bu,\varphi) = \bu(0)-\bu_0$.

	From Theorem~\ref{theorem:timeexist}, we see that given $\varphi \in \Phi_{ad}$ the equation $\mathcal{G}(\bu,\varphi) = 0$ admits a unique weak solution $\bu = \ST(\varphi) \in W_T^2(\bV)$. We employ such operators to establish the differentiability of the time-dependent design-to-state operator via the implicit function theorem.
	
	\begin{proposition}\label{prop:frech-design}
		Suppose that Assumption~\ref{objassump}\ref{en-assump-obj-ii} and the assumptions in Proposition~\ref{prop:timestrong} hold. Then the time-dependent design-to-state operator $\ST:\Phi_{ad}\to W_T^2(\bV)$ is, at least first-order, Fr{\'e}chet differentiable. Its derivative at $\varphi\in\Phi_{ad}$ in the direction $\delta\varphi\in\Phi_{ad}$ is the weak solution $\bud = \ST'(\varphi)\delta\varphi \in W_T^2(\bV)$ of \eqref{poreq:lintimeNS}
		with $\bu_1 = \bu_2 = \ST(\varphi)$, $\bF = -\alpha_\varepsilon'(\varphi)\delta\varphi\, \ST(\varphi) $, and $\bud_0 = 0$.
	\end{proposition}
	
	\begin{proof}
		The differentiability of $\mathcal{G}: W_T^2(\bV)\times \Phi_{ad} \to L^2(I_T;\bV^*)\times \bH$ follows from the fact that it is at most of order two with respect to the variable $\bu\in W_T^2(\bV)$ and from  the Fr{\'e}chet differentiability of $\alpha_\varepsilon$.
		Now, let $(\bu^0,\varphi^0)\in W_T^2(\bV) \times\Phi_{ad}$ such that $\mathcal{G}(\bu^0,\varphi^0) = 0$ in $L^2(I_T;\bV^*)\times \bH$, which also implies that $\bu^0 = \ST(\varphi^0)$. 
		
		Note that, according to Theorem~\ref{theorem:lintimeexist}, the operator $\frac{\partial}{\partial\bu}\mathcal{G}(\bu^0,\varphi^0)\in \mathcal{L}(W^2_T(\bV), L^2(I_T;\bV^*)\times\bH  )$ whose components are defined as 
		\begin{align*}
			\int_0^T\left\langle \frac{\partial}{\partial\bu}\mathcal{G}_1(\bu^0,\varphi^0) \delta\bu, \bpsi \right\rangle_{\bV} \du t 
			= \int_0^T \langle \partial_t \delta\bu,\bpsi \rangle_{\bV} + a_{\varphi^0}(\delta\bu,\bpsi) + b(\delta\bu,\bu^0,\psi) + b(\bu^0,\delta\bu,\psi)\du t
		\end{align*}
		and $\frac{\partial}{\partial\bu}\mathcal{G}_2(\bu^0,\varphi^0) \delta\bu = \delta\bu(0)$ is an isomorphism. From implicit function theorem (cf. \cite{pata2019}) one finds a neighborhood $\mathcal{O}_{\bu}\times\mathcal{O}_\varphi \subset W_T^2(\bV)\times \Phi_{ad}$ of $(\bu^0,\varphi^0)$ and a differentiable operator $\mathcal{S}:\mathcal{O}_{\varphi}\to \mathcal{O}_{\bu} $ such that $\mathcal{G}(\mathcal{S}(\varphi),\varphi) = 0$ for any $\varphi \in \mathcal{O}_\varphi$. 
		By definition of $\ST$, we see that $\ST = \mathcal{S}$ in $\mathcal{O}_{\varphi}$ and $\ST$ is differentiable in $\mathcal{O}_{\varphi}$, and due to the arbitrary nature of $(\bu^0,\varphi^0)$ we see that $\ST$ is differentiable in $\Phi_{ad}$.

		To calculate its derivative, let $\varphi\in\Phi_{ad}$ and use chain rule to $\mathcal{G}(\ST(\varphi),\varphi) = 0$ so that
		\begin{align*}
			\frac{\partial}{\partial\bu}\mathcal{G}(\ST(\varphi),\varphi)[ \ST'(\varphi)\delta \varphi ] + \frac{\partial}{\partial\varphi}\mathcal{G}(\ST(\varphi),\varphi)\delta\varphi = 0.
		\end{align*}
		We also compute the components of $\frac{\partial}{\partial\varphi}\mathcal{G}(\ST(\varphi),\varphi)\delta\varphi$ as $\frac{\partial}{\partial\varphi}\mathcal{G}_1(\ST(\varphi),\varphi)\delta\varphi = \alpha_\varepsilon'(\varphi)\delta\varphi\, \ST(\varphi)$ and $\frac{\partial}{\partial\varphi}\mathcal{G}_2(\ST(\varphi),\varphi)\delta\varphi = 0$. From which, we infer that $$\bud = \ST'(\varphi)\delta\varphi = - \left[ \frac{\partial}{\partial\bu}\mathcal{G}(\ST(\varphi),\varphi) \right]^{-1}\frac{\partial}{\partial\varphi}\mathcal{G}(\ST(\varphi),\varphi)\delta\varphi$$ indeed solves \eqref{poreq:lintimeNS} with $\bu_1 = \bu_2 = \ST(\varphi)$, $\bF = -\alpha_\varepsilon'(\varphi)\delta\varphi\, \ST(\varphi) $, and $\bud_0 = 0$.
	\end{proof}
	Our next aim is to derive an adjoint equation, in order to derive first order optimality conditions.
	One of the operators that we touched on in the previous proof is the operator 
	\begin{align*}
		\left[ \frac{\partial}{\partial\bu}\mathcal{G}(\ST(\varphi),\varphi) \right]^{-1}:L^2(I_T;\bV^*)\times \{0\}_{\bH}\to W^2_0(\bV),  
	\end{align*}
	where $W^2_0(\bV):= \{\bv\in W_T^2(\bV) : \bv(0) = 0 \text{ in }\bH\}$, which according to Theorem~\ref{theorem:lintimeexist} is a bounded linear operator. This allows us to define the adjoint operator 
	\begin{align*}
		\left[ \frac{\partial}{\partial\bu}\mathcal{G}(\ST(\varphi),\varphi) \right]^{-*}:W^2_0(\bV)^*\to L^2(I_T;\bV)\times \{0\}_{\bH} .  
	\end{align*}
	The following lemma characterizes the action of the adjoint operator $\left[ \frac{\partial}{\partial\bu}\mathcal{G}(\ST(\varphi),\varphi) \right]^{-*}$.

	\begin{lemma}
		Let $\bG\in L^2(I_T;\bH)\subset W^2_0(\bV)^*$ and $\bua\in L^2(I_T;\bV)$ be such that $(\bua,0) = \left[ \frac{\partial}{\partial\bu}\mathcal{G}(\ST(\varphi),\varphi) \right]^{-*}\bG \in L^2(I_T;\bV)\times \{0\}_{\bH}$. Then, $\bua\in L^2(I_T;\bV)$ solves \eqref{poreq:adjtimeNS} with $\bG$ on the right-hand side, $\bu_1=\bu_2=\ST(\varphi)$, and $\bua(T) = 0$.
	\end{lemma}
	\begin{proof}   
		Let $\bud\in W^2_0(\bV)$ be arbitrary and $(\bF,0) = \frac{\partial}{\partial\bu}\mathcal{G}(\ST(\varphi),\varphi) \bud$, and use the notation $\bu = \ST(\varphi)$.
		We have by integration by parts
		\begin{align*}
			&\left\langle (\bua,0), (\bF,0)  \right\rangle_{ L^2(I_T;\bV^*)\times \{0\}_{\bH} }\\
			&= \left\langle (\bua,0), \frac{\partial}{\partial\bu}\mathcal{G}(\ST(\varphi),\varphi) \bud \right\rangle_{ L^2(I;\bV^*)\times \{0\}_{\bH} }  = \int_0^T \langle \partial_t\bud,\bua \rangle + a_\varphi(\bud,\bua) + b(\bud,\bu,\bua) + b(\bu,\bud,\bua) \du t.
		\end{align*}
		On the other hand we have
		\begin{align*}
			\left\langle (\bua,0), (\bF,0)  \right\rangle_{ L^2(I_T;\bV^*)\times \{0\}_{\bH} } & = \left\langle \left[ \frac{\partial}{\partial\bu}\mathcal{G}(\ST(\varphi),\varphi) \right]^{-*}\bG, (\bF,0)  \right\rangle_{ L^2(I;\bV^*)\times \{0\}_{\bH} }\\
			& = \left\langle \bG, \left[ \frac{\partial}{\partial\bu}\mathcal{G}(\ST(\varphi),\varphi) \right]^{-1}(\bF,0)  \right\rangle_{ L^2(I_T;\bH)} = \int_0^T\int_\Omega \bG\cdot \bud \du x\du t.
		\end{align*}

		To be able to apply integration by parts with respect to the time integral, let us check the regularity of $\partial_t\bua$. 
		Let us momentarily take $\bud\in C_0^\infty(I_T;\bV)$ so that 
		\begin{align*}
			\int_0^T \langle \bud,\partial_t\bua \rangle \du t =  \int_0^T a_\varphi(\bud,\bua) + b(\bud,\bu,\bua) + b(\bu,\bud,\bua) \du t - \int_0^T\int_\Omega \bG\cdot \bud\du x \du t.
		\end{align*}
		Invoking the regularity of the state variable $\bu = \ST(\varphi)$ and $\bG\in L^2(I_T;\bH)$, we see that the map
		\begin{align*}
			\bud \mapsto \int_0^T a_\varphi(\bud,\bua) + b(\bud,\bu,\bua) + b(\bu,\bud,\bua) \du t - \int_0^T\int_\Omega \bG\cdot \bud\du x \du t
		\end{align*}
		belongs to $L^2(I_T;\bV^*)$, and consequently $\partial_t \bua\in L^2(I_T;\bV^*) $ and $\bua\in W^2(\bV)$. 
		
		We can thus employ integration by parts, so that if $\bua\in L^2(I_T;\bV)$ and by adding the term $(\dive \bud,\pa) = 0$ to be able to recover the pressure, we get
		\begin{align*}
			{(\bud(T),\bua(T))} + \int_0^T \int_\Omega \bud\cdot\left( -\partial_t\bua + \alpha_\varepsilon(\varphi)\bua -\mu\Delta\bua + (\nabla\bu)^{\top}\bua - (\bu\cdot\nabla)\bua + \nabla\pa\,\right) \du x\du t&\\ = \int_0^T\int_\Omega \bG\cdot \bud\du x \du t&
		\end{align*}
		The arbitrary nature of $\bud\in W_0^2(\bV)$ proves our claim.
	\end{proof}

	From the operators and the discussions above, we are afforded to formulate the necessary optimality condition for a local solution---defined as an element $\varphi^T\in\Phi_{ad}$ such that $J_T(\varphi^T)\le J_T(\varphi)$ for all $\varphi\in \Phi_{ad}$ satisfying $\|\varphi^T-\varphi \|_{H^1(\Omega)}\le \delta$ for some $\delta>0$--- of the optimization problem \eqref{reduced:objphasetime}.
	
	\begin{theorem}
		Suppose that Assumption \ref{objassump} and the assumptions in Proposition \ref{prop:frech-design} hold.
		If $\varphi^T\in \Phi_{ad}$ is a minimizer for $J_T$, then 
		\begin{align}\label{variationalineq:time}
			\begin{aligned}
				0\le&\, \frac{\gamma}{2c_0}\int_\Omega \left[\varepsilon\nabla\varphi^T\cdot \nabla(\varphi - \varphi^T) + \frac{1}{\varepsilon}\Psi_0'(\varphi^T) (\varphi-\varphi^T)\right]\du x \\
				&\, + \frac{1}{T}\int_{\Omega_T} \left({\mathcal{B}}_\varepsilon(\varphi^T)|\bu|^2 - \mathcal{A}_\varepsilon(\varphi^T)\bu\cdot\bua + \frac{1}{2}|\bu-\bu_d|^2\one_{\omega}\right) (\varphi-\varphi^T) \du x\du t
			\end{aligned}\qquad \forall \varphi\in\Phi_{ad}
		\end{align}
		where $\bu = \ST(\varphi^T)$ and $\bua \in \bW^{2,1}_{2,T}$ is such that $(\bua,0) = \left[ \frac{\partial}{\partial\bu}\mathcal{G}(\ST(\varphi),\varphi) \right]^{-*}\left( ({\varphi^T + 1}) \left(\bu - \bu_d\right)\one_\omega + 2\alphat_\varepsilon(\varphi^T)\bu\right)$, i.e., $\bua$ solves the adjoint equation.
	\end{theorem}
	\begin{proof}
		
		We write the objective functional as 
		\begin{align*}
			J_T(\varphi) =  \mathcal{J}_T(\varphi) + \frac{\gamma}{2c_0\varepsilon}\int_\Omega I_{[-1,1]}(\varphi) \du x =:  \mathcal{J}_T(\varphi) + \boldsymbol{I}_{[-1,1]}(\varphi), 
		\end{align*}
		where $\mathcal{J}_T$ is Fr{\'e}chet differentiable in $\Phi_{ad}$.        
		The differentiability of $\mathcal{J}_T$ follows from the differentiability of the $L^2$-norms and the differentiability of $\alpha_\varepsilon$, and $\ST$ from Proposition~\ref{prop:frech-design}.
		
		Since $\varphi^T\in\Phi_{ad}$ is a minimizer, then $0\in \partial J_T(\varphi^T)$.
		Hence, we have $-\mathcal{J}_T'(\varphi^T) \in \partial\boldsymbol{I}_{[-1,1]}(\varphi^T)$.
		By definition, we have $\boldsymbol{I}_{[-1,1]}(\varphi^T) - \boldsymbol{I}_{[-1,1]}(\varphi) \le \mathcal{J}_T'(\varphi^T)(\varphi-\varphi^T)$ for arbitrary $\varphi \in \Phi_{ad}$.
		Since $\varphi^T,\varphi\in\Phi_{ad}$, i.e., $|\varphi^T|\le 1$ and  $|\varphi|\le 1$ a.e. in $\Omega$, we have $\boldsymbol{I}_{[-1,1]}(\varphi^T) - \boldsymbol{I}_{[-1,1]}(\varphi) = 0$.

		Let $\delta\varphi\in \{- \varphi^T+ \Phi_{ad}\}$ be arbitrary. 
		We use chain rule to get
		\begin{align*}
			\mathcal{J}_T'(\varphi^T)\delta\varphi =&  \frac{1}{T}\int_{\Omega_T}2\left[ \left( \frac{\varphi^T + 1}{2} \right) (\bu-\bu_d)\one_\omega + \alphat_\varepsilon(\varphi^T) \bu \right]\cdot \ST'(\varphi^T)\delta\varphi  \du x\du t\\
			&\ + \frac{1}{T}\int_{\Omega_T}\frac{\delta\varphi}{2}  |\bu-\bu_d|^2 \one_\omega + \alphat_\varepsilon'(\varphi^T)\delta\varphi|\bu|^2 \du x\du t\\
			&\ + \frac{\gamma}{2c_0} \int_\Omega \varepsilon\nabla\varphi^T\cdot\nabla\delta\varphi + \frac{1}{\varepsilon}\Psi_0'(\varphi^T)\delta\varphi \du x.
		\end{align*} 
		Note that by definition $\ST'(\varphi^T)\delta\varphi = \left[ \frac{\partial}{\partial\bu}\mathcal{G}(\ST(\varphi^T),\varphi^T) \right]^{-1}( -\alpha_\varepsilon'(\varphi^T)\delta\varphi \bu, 0)$. 
		
		From the definition of the adjoint operator $\left[ \frac{\partial}{\partial\bu}\mathcal{G}(\ST(\varphi^T),\varphi^T) \right]^{-*}$, we can rewrite the first integral as 
		\begin{align*}
			&\frac{1}{T}\int_{\Omega_T}2\left[ \left( \frac{\varphi^T + 1}{2} \right) (\bu-\bu_d)\one_\omega + \alphat_\varepsilon(\varphi^T) \bu \right]\cdot \ST'(\varphi^T)\delta\varphi  \du x\du t\\ & = \frac{1}{T}\left\langle ({\varphi^T + 1}) (\bu-\bu_d)\one_\omega + \alphat_\varepsilon(\varphi^T) \bu, \left[ \frac{\partial}{\partial\bu}\mathcal{G}(\ST(\varphi^T),\varphi^T) \right]^{-1}( -\alpha_\varepsilon'(\varphi^T)\delta\varphi \bu, 0)  \right\rangle_{W^2_0(\bV)}\\
			& = \frac{1}{T}\left\langle (\bua,0),( -\alpha_\varepsilon'(\varphi^T)\delta\varphi \bu, 0)   \right\rangle_{L^2(I_T;\bV^*)\times\{0\}_{\bH} } = \frac{1}{T}\int_{\Omega_T} \bua\cdot (-\alpha_\varepsilon'(\varphi^T)\delta\varphi \bu) \du x\du t.
		\end{align*}

		After rearrangement, by invoking Assumption \ref{porousassum}(ii) and by using $\mathcal{J}_T'(\varphi^T)(\varphi-\varphi^T)\ge 0$, we infer that
		\begin{align*}
			0 &\le \mathcal{J}_T'(\varphi^T)(\varphi-\varphi^T) = \frac{\gamma}{2c_0} \int_\Omega \varepsilon\nabla\varphi^T\cdot\nabla(\varphi- \varphi^T) + \frac{1}{\varepsilon}\Psi_0'(\varphi^T)(\varphi- \varphi^T) \du x \\
			&\ + \frac{1}{T}\int_{\Omega_T} \left( \mathcal{B}_{\varepsilon}(\varphi^T)|\bu|^2 - \mathcal{A}_{\varepsilon}(\varphi^T)\bu\cdot\bua + \frac{1}{2}|\bu-\bu_d|^2\one_\omega \right)(\varphi- \varphi^T) \du x\du t .
		\end{align*} 
		Lastly, the property that $\bua\in \bW^{2,1}_{2,T}$ follows from the facts that $({\varphi^T + 1}) \left(\bu - \bu_d\right)\one_\omega + 2\beta_\varepsilon(\varphi^T)\bu \in L^2(I_T;\bL^2(\Omega)^2)$ and $\bu\in\bW^{2,1}_{2,T} $.
	\end{proof}

	In the case of the stationary problem, the first order optimality condition has been well-studied. In fact we refer to \cite{garcke2018}, where the authors considered a generalized objective functional with integral state constraints. Adapting their computations to our case accordingly, i.e. the tracking type functional and removing the integral constraints, we arrive at the following optimality system.
	
	\begin{theorem}
		Suppose that Assumption \ref{objassump} and the assumptions in Proposition \ref{prop:statstrong} hold. If $\varphi^s\in \Phi_{ad}$ is a minimizer for $J_s$, then 
		\begin{align}\label{variationalineq:stat}
			\begin{aligned}
				0\le&\, \frac{\gamma}{2c_0}\int_\Omega \left[\varepsilon\nabla\varphi^s\cdot \nabla(\varphi - \varphi^s) + \frac{1}{\varepsilon}\Psi_0'(\varphi^s) (\varphi-\varphi^s)\right]\du x \\
				&\, + \int_{\Omega} \left({\mathcal{B}}_\varepsilon(\varphi^s)|\bv|^2 - \mathcal{A}_\varepsilon(\varphi^s)\bv\cdot\bva + \frac{1}{2}|\bv-\bu_d|^2\one_{\omega}\right) (\varphi-\varphi^s) \du x
			\end{aligned}\qquad \forall \varphi\in\Phi_{ad}
		\end{align}
		where $\bv = \Ss(\varphi^s)$ and $\bva \in \bV$ is the weak solution to \eqref{poreq:adjstatNS} with $\bv_1=\bv_2 = \bv$ and $\bG_s = ({\varphi^s + 1}) \left(\bv - \bu_d\right)\one_\omega + 2\alphat_\varepsilon(\varphi^s)\bv$.
	\end{theorem}
	
	Note that to be able to come up with the variational inequality \eqref{variationalineq:stat}, just as in the case of the time-dependent problem, we split the objective functional $J_s$ into its differentiable and singular parts. 
	
	Using analogous notation for the stationary problem, we can write the variational inequalities \eqref{variationalineq:time} and \eqref{variationalineq:stat} as follows:
	\begin{align*}
		\mathcal{J}_{\times}'(\varphi^{\times})(\varphi - \varphi^\times) \ge 0\quad \forall\varphi\in\Phi_{ad},
	\end{align*}
	where $\times\in\{s,T \}$ and $\varphi^\times\in\Phi_{ad}$ is a minimizer for $J_\times$.


	\section{Long time behavior of minima}
	\label{sec:LongTimeBehavior}
	
	The results above have now put us in a position to prove our main result. In Theorem~\ref{thm:asymp} we provide an asymptotic bound for the gap between the evaluation of the objective functions of the time-dependent and the stationary optimization problems with their respective optimal solutions. 
	The steps required to prove such result utilize the established energy estimates of the involved state equations in Section~\ref{sec:solGovEq}.
	
	A consequence of Theorem~\ref{thm:asymp} is that it provides us with a way to establish the uniform (with respect to the final time $T$) boundedness of solutions of the time-dependent optimization problem. This allows us to prove in Theorem~\ref{thm:phiInftyIsMin} that a sequence of solutions for \eqref{objphasetime} converges to a minimizer of \eqref{objphasestat}. 
	This proof is more nuanced than the proof of the asymptotic bound as we shall employ the first order optimality condition for the time-dependent problem.

	\begin{theorem} 
		\label{thm:asymp}
		Suppose that the assumptions in Theorem \ref{thm:minexist} hold. Let $T>0$, and $\varphi^T\in \Phi_{ad}$ and $\varphi^s\in \Phi_{ad}$ be  global minimizers of \eqref{reduced:objphasetime} and \eqref{reduced:objphasestat}, respectively. 
		Then it holds
		\begin{align}
			|J_T(\varphi^T) - J_s(\varphi^s)| \le \left(  \frac{C_1}{T} + \frac{C_2}{\sqrt{T}} \right) \label{goal}
		\end{align}
		for some constants $C_1,C_2>0$ independent of $T$ such that, 
        \jshsb{for some $C>0$ also independent of $T$,
        \begin{align*}
            C_1\le C(\sqrt{K}+\|\bu_0\|_{\bH} + \|\bu_d\|_{\bL^2(\omega)} + \|\bf_s\|_{\bL^2})\quad\text{ and }\quad C_2 \le C\|\bf_s\|_{\bL^2}.
        \end{align*}}
	\end{theorem}
	\begin{proof}
		Let $\bu^s = \ST(\varphi^s)$ and $\bv^s = \Ss(\varphi^s)$. Since $\varphi^T$ is a minimizer of $J_T$ we have with a generic constant $C>0$ independent of $T$
		\begin{align}
			& J_T(\varphi^T) - J_s(\varphi^s) \le J_T(\varphi^s) - J_s(\varphi^s) \nonumber\\
			& = \frac{1}{T}\left[ \int_{\Omega_T} \left(\frac{\varphi^s + 1}{2} \right)(\bu^s - \bv^s)\cdot(\bu^s+\bv^s-2\bu_d) \one_\omega \du x\du t + \int_{\Omega_T} \alphat_\varepsilon(\varphi^s)(\bu^s - \bv^s)\cdot(\bu^s+\bv^s)  \du x\du t \right]\nonumber\\
			& \le \frac{C}{T} \int_0^T  \big\{ 
            \left(\|\bu^s \|_{\bH} + \|\bv^s \|_{\bH} + 2\|\bu_d \|_{\bL^2(\omega)} \right) + \| \alphat_\varepsilon(\varphi^s) \|_{L^\infty}\!\left(\|\bu^s \|_{\bH} + \|\bv^s \|_{\bH} \right) \big\} 
            \|\bu^s - \bv^s \|_{\bH}  \du t\nonumber\\
            %
            %
            & \jshsb{
            \leq \frac{C}{T}\max\{1,\overline \beta_\epsilon\}\left( \|\bu^s\|_{L^\infty(\bH)}  + \|\bv^s \|_{\bV} + \|\bu_d \|_{\bL^2(\omega)}\right) \|\bu^s - \bv^s\|_{L^1(\bH)}
            }\nonumber\\
            &\jshsb{\leq \frac{C}{T}\left( \|\bf\|_{L^2(\bL^2)} + \|\bu_0\|_{\bH} + \|\bf_s\|_{\bL^2} + \|\bu_d \|_{\bL^2(\omega)} \right)\|\bw^s\|_{L^1(\bH)}\label{gapest:1}
            }
		\end{align}
		where we used \eqref{energy:time} and \eqref{energy:stat}$\bw^s := \bu^s - \bv^s$.
        The next task is to obtain an estimate for the term $\|\bw^s \|_{L^1(\bH)}$. To do so, we note that the element $\bw^s$ satisfies the equations
		\begin{align}\label{poreq:ws}
			&\left\{
			\begin{aligned}
				\partial_t\bw^s + \alpha_\varepsilon(\varphi)\bw^s - \mu\Delta\bw^s + (\bu^s\cdot\nabla)\bw^s + (\bw^s\cdot\nabla)\bv^s + \nabla q & = \bf - \bf_s &&\text{ in }\Omega_T,\\
				\dive\bw^s& = 0 && \text{ in }\Omega_T,\\
				\bw^s & = 0 &&\text{ on }\Sigma_T,\\
				\bw^s(0) & = \bu_0 - \bv^s &&\text{ in }\Omega.
			\end{aligned}
			\right.
		\end{align}
		Note that \eqref{poreq:ws} is of the form \eqref{poreq:lintimeNS} with $\bu_1 = \bv^s$, $\bu_2 = \bu^s$, $\bF = \bf - \bf_s$, and $\bud_0 = \bu_0-\bv^s$. 
		By taking $\bz = \frac{c_P}{\mu}\bf_s$ in Proposition~\ref{prop:decaylin} we get
		\begin{align*}
			\|\bw^s(t)\|_{\bH}^2 \le e^{A(t)}\|\bud_0\|_{\bH}^2 + \frac{2c_P^2}{\mu}\int_0^te^{A(t)-A(\tau)}\|\bF(\tau)\|_{\bL^2}^2 \du \tau,
		\end{align*}
		where
		\begin{align*}
			A(t) = -\left( \frac{\mu}{c_P^2} - \frac{2c_L^2c_P^2}{\mu^3}\|\bf_s\|_{\bL^2} \right)t = -\varsigma t,
		\end{align*}
		Assumption~\ref{sourceassump} thus imply that
		\begin{align}
            \label{wL2}
            \begin{aligned}
    			\|\bw^s(t)\|_{\bH}^2 & \le e^{-\varsigma t}\|\bu_0 - \bv^s \|_{\bL^2}^2 + \frac{2c_p^2}{\mu}\int_0^t e^{-\varsigma(t-\tau)}\|\bf(\tau) - \bf_s \|_{\bL^2}^2 \du\tau\\ & \le e^{-\varsigma t}\left( \|\bu_0 - \bv^s \|_{\bL^2}^2 + \frac{2c_p^2}{\mu}K \right).
            \end{aligned}
		\end{align}
        \jshsb{Taking the integral of the square root of \eqref{wL2} yields
        \begin{align*}
            \|\bw^s\|_{L^1(\bH)} & \le \frac{2}{\varsigma}\left(\|\bu_0 - \bv^s \|_{\bL^2}^2 + \frac{2c_p^2}{\mu}K \right)^{\!1/2}\left( 1- e^{-\frac{\varsigma}{2} T}  \right) \le \frac{2}{\varsigma}\left(\|\bu_0 - \bv^s \|_{\bL^2}^2 + \frac{2c_p^2}{\mu}K \right)^{\!1/2} =: C_{D}.
        \end{align*}
        }
        
		Due to Assumption~\ref{statdatassum} we note that $\mu^4 - 2c_L^2c_P^4\|\bf_s\|_{\bL^2}^2 > 0$, which implies the positivity of $\varsigma$, hence, the uniform boundedness of
        $\|\bw^s\|_{L^1(\bH)}$.
		
		Plugging the attained estimate for $\bw^s$ into \eqref{gapest:1}, together with \eqref{energy:time} and \eqref{energy:stat}, we infer that
        \jshsb{
		\begin{align*}
			J_T(\varphi^T) - J_s(\varphi^s)
            &\leq \frac{C}{T}\left( \|\bf\|_{L^2(\bL^2)} + \|\bu_0\|_{\bH} + \|\bf_s\|_{\bL^2} + \|\bu_d \|_{\bL^2(\omega)} \right)\|\bw^s\|_{L^1(\bH)}
            \\
            &\leq \frac{C}{T}\left( \|\bf\|_{L^2(\bL^2)} + \|\bu_0\|_{\bH} + \|\bf_s\|_{\bL^2} + \|\bu_d \|_{\bL^2(\omega)} \right)
		\end{align*} 
		We note that the constant $C>0$ above is independent of $T$.
		Lastly, from \eqref{connect:ffs} we have
		\begin{align}\label{ineq:t-s}
            \begin{aligned}
    			J_T(\varphi^T) - J_s(\varphi^s)&\le  \frac{C}{T}
    			\left( \sqrt{K} + \|\bu_0\|_{\bH} + \|\bu_d \|_{\bL^2(\omega)} + \|\bf_s \|_{\bL^2}\left(1+ \sqrt{T} \right) \right)\\
                &\le 
                \frac{C}{T}\left(\sqrt{K}+ \|\bu_0\|_{\bH} + \|\bu_d \|_{\bL^2(\omega)} + \|\bf_s \|_{\bL^2}\right) 
                 + \frac{C}{\sqrt{T}}\|\bf_s \|_{\bL^2} 
             \end{aligned}
		\end{align}}
		%
		Starting with 
		\begin{align*}
			J_s(\varphi^s) - J_T(\varphi^T)
			\le J_s(\varphi^T) - J_T(\varphi^T)
		\end{align*}
		and  using similar arguments as above it holds
		\begin{align*}
			J_s(\varphi^s) - J_T(\varphi^T)
			\le \left( \frac{C_1}{T} + \frac{C_2}{\sqrt{T}} \right),
		\end{align*}
		which concludes the proof.
	\end{proof}
	
	\begin{remark}
		\label{rm:JTs_m_JsT-bounded}
		We note that the computations in the proof above estimated the terms $J_T(\varphi^s) - J_s(\varphi^s)$ and $J_T(\varphi^T) - J_s(\varphi^T)$ to obtain the desired bound. Using similar line of arguments, we can show that
		\begin{align*}
			\max\{ |J_T(\varphi^s) - J_s(\varphi^s)|,|J_T(\varphi^T) - J_s(\varphi^T)| \}\le \left( \frac{C_1}{T} + \frac{C_2}{\sqrt{T}} \right).
		\end{align*}
		This is done, for example, by taking the absolute value of the terms in the second line in \eqref{gapest:1} and continuing as is resulting to the estimate for $|J_T(\varphi^s) - J_s(\varphi^s)|$.
        
	\end{remark}
    \jshsr{
    \begin{remark}
      In certain optimal control frameworks (cf. \cite{PSJMFM2021,jork2025}), monitoring the state at the terminal time is of particular interest. Accordingly, if instead of considering the time-average tracking term, we replace it with final time tracking of the form $\int_\Omega \frac{\varphi+1}{2} |\bu(T)-\bu_d|^2 \chi_\omega \du x$, following the computations above--specifically \eqref{wL2}, we get the bound 
      \begin{align*}
        |J_T(\varphi^T) - J_s(\varphi^s)| \leq e^{-\frac{\varsigma T}{2}}(C_1+C_2\sqrt{T}) + \left( \frac{C_3}{T} + \frac{C_4}{\sqrt{T}} \right), 
      \end{align*}
      where the first term, involving the exponential originates from the final time tracking, and the second term is from the penalty for the porous medium approximation.
    \end{remark}
    }
	\begin{remark}
		The result of Theorem~\ref{thm:asymp} 
		justifies to approximate a solution to the time-dependent problem \ref{reduced:objphasetime} by a solution to the stationary problem \ref{reduced:objphasestat}, which in practice should be significantly cheaper to calculate. 
		On the other hand, it  gives rise to the question of whether a sequence of solutions to the time-dependent problem for $T\to \infty$ converges to a solution of the stationary problem.
	\end{remark}
	If we consider a sequence $(T_n)\subset \mathbb{R}_{+}$ such that $T_n\to \infty$ as $n\to \infty$, we take a sequence $(\varphi^{T_n}) \subset \Phi_{ad}$ of minimizers of $J_{T_n}$. From \eqref{ineq:t-s}, we see that
	\begin{align*}
		\frac{\gamma \varepsilon}{4c_0}\|\nabla\varphi^{T_n}\|_{\bL^2}^2 \le J_{T_n}(\varphi^{T_n}) \le C\left( \frac{1}{T_n} + \frac{1}{\sqrt{T_n}} \right) + J_s(\varphi^s) \le C +J_s(\varphi^s).
	\end{align*}
	We  thus find $\varphi^\infty \in\Phi_{ad}$ and extract a subsequence of $(\varphi^{T_n})$, which we denote in the same manner, such that the following convergences hold
	\begin{align}
		&\varphi^{T_n} \rightharpoonup \varphi^{\infty} \text{ in }H^1(\Omega),\label{h1conv}\\
		&\varphi^{T_n} \to \varphi^{\infty} \text{ in }L^p(\Omega) \text{ for } 2\le p < \infty \text{ and a.e. in } \Omega,\label{lpconv}\\
		&\varphi^{T_n} \ws \varphi^{\infty} \text{ in }L^\infty(\Omega).\label{linfconv}
	\end{align}
	The natural question that follows is: What role does the element $\varphi^\infty \in \Phi_{ad}$ play? Theorem~\ref{thm:phiInftyIsMin} establishes that $\varphi^\infty$ is, in fact, a minimizer of $J_s$.

	\begin{theorem}
		\label{thm:phiInftyIsMin}
		The function $\varphi^\infty\in\Phi_{ad}$ is a global minimizer of $J_s$. 
	\end{theorem}

	\begin{proof}
		Since we have established the existence of a minimizer $\varphi^s\in\Phi_{ad}$ for $J_s$, we show that $\varphi^\infty$ is also a minimizer by proving that $J_s(\varphi^s) = J_s(\varphi^\infty)$. 
		
		With the aid of Theorem~\ref{thm:asymp} and Remark~\ref{rm:JTs_m_JsT-bounded}, we find that
		for any $T_n$ with corresponding global minimizer $\varphi^{T_n}$ of $J_{T_n}$ it holds
		\begin{align*}
			|J_s(\varphi^s) - J_s(\varphi^\infty) | & \le |J_s(\varphi^s) - J_{T_n}(\varphi^{T_n}) | + |J_{T_n}(\varphi^{T_n}) - J_s(\varphi^{T_n}) | + |J_s(\varphi^{T_n}) - J_s(\varphi^\infty) | \\
			& \le C\left(\frac{1}{T_n} + \frac{1}{\sqrt{T_n}}\right) + |J_s(\varphi^{T_n}) - J_s(\varphi^\infty) |.
		\end{align*}
		
		The next step is to get proper estimates for the last term. 
		Let $\bv^{T_n} = \Ss(\varphi^{T_n})$ and $\bv^\infty = \Ss(\varphi^\infty)$, we have
		\begin{align*}
			& |J_s(\varphi^{T_n}) - J_s(\varphi^\infty) | = \left| \int_\Omega \left( \frac{\varphi^{T_n} + 1}{2} \right) |\bv^{T_n}-\bu_d|^2\one_\omega \du x  + \int_\Omega \alphat_\varepsilon(\varphi^{T_n})|\bv^{T_n}|^2\du x + \gamma \mathbb{E}_{\varepsilon}(\varphi^{T_n}) \right.\\
			&\hspace{1.35in} - \left. \left[\int_\Omega \left( \frac{\varphi^\infty + 1}{2} \right) |\bv^\infty-\bu_d|^2\one_\omega \du x  + \int_\Omega \alphat_\varepsilon(\varphi^\infty)|\bv^\infty|^2\du x + \gamma \mathbb{E}_{\varepsilon}(\varphi^\infty) \right] \right|\\
			&\le  \left| \int_\Omega \left( \frac{\varphi^{T_n} + 1}{2} \right)( \bv^{T_n} - \bv^\infty )\cdot( \bv^{T_n} + \bv^\infty - 2\bu_d)\one_\omega \du x\right|  + \left|\int_\Omega \left( \frac{\varphi^{T_n} - \varphi^\infty}{2} \right) |\bv^\infty-\bu_d|^2\one_\omega \du x \right|\\
			&\hspace{.15in} + \left| \int_\Omega \alphat_\varepsilon(\varphi^{T_n})( \bv^{T_n} - \bv^\infty )\cdot( \bv^{T_n} + \bv^\infty)\du x\right|  + \left|\int_\Omega \left( \alphat_\varepsilon(\varphi^{T_n}) - \alphat_\varepsilon(\varphi^\infty)\right) |\bv^\infty|^2\du x \right|\\
			&\hspace{.15in}+ \gamma|\mathbb{E}_{\varepsilon}(\varphi^{T_n}) - \mathbb{E}_{\varepsilon}(\varphi^\infty) | \\
			&=: I_1 + I_2 + I_3 + I_4 +I_5.
		\end{align*}
		
		\paragraph*{Treatment of $I_1$--$I_4$.}
		For $I_1$, we note that from Theorem~\ref{theorem:statcontinuity}  we have $\bv^{T_n}\to \bv^\infty$ in $\bL^2(\Omega)$. From this, and from the inequalities \eqref{bound:Phi} and \eqref{energy:stat} we have
		\begin{align*}
			I_1 &\le \| \bv^{T_n} - \bv^{\infty} \|_{\bL^2}\left( c_P(\| \bv^{T_n}\|_{\bV} + \| \bv^\infty\|_{\bV})+ 2\|\bu_d\|_{\bL^2(\omega)}\right)\\
			& \le \| \bv^{T_n} - \bv^{\infty} \|_{\bL^2}\left(\frac{2c_P^2}{\mu}\|\bf_s \|_{\bL^2}+ 2\|\bu_d\|_{\bL^2(\omega)}\right) \to 0 \text{ as }n\to \infty.
		\end{align*}
		Since $\bv^\infty\in \bL^2(\Omega)$ and $\bu_d\in \bL^2(\omega)$, we have $|\bv^\infty - \bu_d|^2\one_\omega \in L^1(\Omega)$. The convergence $\varphi^{T_n} \ws \varphi^{\infty}$ in $L^\infty(\Omega)$, therefore implies that $I_2$ goes to zero as $n\to \infty$. Using same arguments as for $I_1$ and $I_2$, the terms $I_3$ and $I_4$ converge to zero, respectively.
		
		\paragraph*{Treatment of $I_5$.}
		The treatment of the term $I_5$ is more nuanced. In fact, we will need to rely on the first order optimality condition for the minimizer of $J_{T_n}$. For now, the definition of the Ginzburg--Landau energy functional gives us
		\begin{align*}
			|\mathbb{E}_{\varepsilon}(\varphi^{T_n}) - \mathbb{E}_{\varepsilon}(\varphi^\infty) | & \le  \frac{\varepsilon}{4c_0}\left| \int_\Omega |\nabla\varphi^{T_n}|^2 - |\nabla\varphi^\infty|^2 \du x\right| +  \frac{1}{2\varepsilon c_0}\left| \int_\Omega \Psi_0(\varphi^{T_n}) - \Psi_0(\varphi^\infty) \du x  \right|\\
			& \le  \frac{\varepsilon}{4c_0}\left| \int_\Omega |\nabla\varphi^{T_n}|^2 - |\nabla\varphi^\infty|^2 \du x\right| +  \frac{1}{4\varepsilon c_0} \|\varphi^{T_n} - \varphi^\infty\|_{L^2} \|\varphi^{T_n}+\varphi^\infty\|_{L^2} =: J_1 + J_2
		\end{align*}
		Due to \eqref{lpconv}, we see that the $J_2$ goes to zero. For $J_1$, we use \eqref{h1conv}, so that
		\begin{align*}
			J_1 \le  \frac{\varepsilon}{4c_0}\left| \int_\Omega (\nabla\varphi^{T_n} - \nabla\varphi^\infty)\cdot\nabla\varphi^{T_n}\du x\right| +  \frac{\varepsilon}{4c_0}\underbrace{\left| \int_\Omega (\nabla\varphi^{T_n} - \nabla\varphi^\infty)\cdot\nabla\varphi^\infty \du x\right|}_{\to 0\text{ as }n\to \infty}.
		\end{align*}
		Here the second term tends to zero by the weak convergence $\varphi^{T_n} \rightharpoonup \varphi^{\infty}$ in $H^1(\Omega)$.
		
		Since $\varphi^{T_n}$ is a minimizer for $J_{T_n}$, it satisfies the necessary optimality condition \eqref{variationalineq:time}. Taking $\varphi = \varphi^\infty$, we see that
		\begin{align*}
			\frac{\gamma\varepsilon}{4c_0}\int_\Omega \nabla\varphi^{T_n}&\cdot  (\nabla\varphi^{T_n} - \nabla\varphi^\infty)\du x\\
			\le& \frac{1}{T_n}\left[\int_{\Omega_{T_n}} \left( \frac{\varphi^\infty - \varphi^{T_n} }{2} \right) |\bu^{T_n}-\bu_d|^2 \one_\omega + {\mathcal{B}}_\varepsilon(\varphi^{T_n})(\varphi^\infty - \varphi^{T_n})|\bu^{T_n}|^2\du x\du t\right.\\
			& \left.- \int_{\Omega_{T_n}}  {\mathcal{A}}_\varepsilon(\varphi^{T_n})(\varphi^\infty - \varphi^{T_n})\bu^{T_n}\cdot\bua^{T_n} \du x\du t\right] + \frac{\gamma}{2\varepsilon c_0}\int_\Omega \Psi_0'(\varphi^{T_n})(\varphi^\infty - \varphi^{T_n}) \du x\\
			\le & \frac{1}{T_n}\left[\left|\int_{\Omega_{T_n}}  \left( \frac{\varphi^\infty - \varphi^{T_n} }{2} \right) |\bu^{T_n}-\bu_d|^2 \one_\omega\du x\du t \right| + \left|\int_{\Omega_{T_n}}  {\mathcal{B}}_\varepsilon(\varphi^{T_n})(\varphi^\infty - \varphi^{T_n})|\bu^{T_n}|^2\du x\du t \right| \right.\\
			& \left.+ \left|\int_{\Omega_{T_n}}  {\mathcal{A}}_\varepsilon(\varphi^{T_n})(\varphi^\infty - \varphi^{T_n})\bu^{T_n}\cdot\bua^{T_n} \du x\du t \right| \right] + \frac{\gamma}{2\varepsilon c_0}\left| \int_\Omega \Psi_0'(\varphi^{T_n})(\varphi^\infty - \varphi^{T_n}) \du x\right|\\
			=:& \frac{1}{T_n}(J_{1,1} + J_{1,2} + J_{1,3}) 
			+ \frac{\gamma}{2\varepsilon c_0}J_{1,4},
		\end{align*}
		where $\bu^{T_n} = \ST(\varphi^{T_n})$ and $(\bua^{T_n},0) = \left[ \frac{\partial}{\partial\bu}\mathcal{G}(\bu^{T_n},\varphi^{T_n}) \right]^{-*}\left( ({\varphi^{T_n} + 1}) \left(\bu^{T_n} - \bu_d\right)\one_\omega + 2\alphat_\varepsilon(\varphi^{T_n})\bu^{T_n}\right)$.
		We use H{\"o}lder inequality, the embedding $H^1\hookrightarrow L^4$, \eqref{energy:time}, \eqref{connect:ffs} to get the following estimate for $J_{1,1}$ using a generic constant $C_{1,1}$.
		\begin{align*}
			\frac{1}{2}J_{1,1} & \le \left|\int_{\Omega_{T_n}}  \left( \frac{\varphi^\infty - \varphi^{T_n} }{2} \right) |\bu^{T_n}|^2 \du x\du t \right| + \left|\int_{\Omega_{T_n}}  \left( \frac{\varphi^\infty - \varphi^{T_n} }{2} \right) |\bu_d|^2 \one_\omega\du x\du t \right|\\
			& \le C_{1,1}\|\varphi^\infty - \varphi^{T_n} \|_{L^2}\left|\int_0^{T_n} \|\bu^{T_n}\|_{\bL^4}^2 \du t \right| + T_n \left|\int_\Omega \left( \frac{\varphi^\infty - \varphi^{T_n} }{2} \right) |\bu_d|^2 \one_\omega\du x\right|\\
			& \le C_{1,1}\|\varphi^\infty - \varphi^{T_n} \|_{L^2}\|\bu^{T_n}\|_{L^2(\bV)}^2 + T_n \left|\int_\Omega \left( \frac{\varphi^\infty - \varphi^{T_n} }{2} \right) |\bu_d|^2 \one_\omega\du x\right|\\
			& \le C_{1,1}\|\varphi^\infty - \varphi^{T_n} \|_{L^2}(\|\bf\|_{L^2(\bL^2)}^2 + \|\bu_0\|_{\bH}^2) + T_n \left|\int_\Omega \left( \frac{\varphi^\infty - \varphi^{T_n} }{2} \right) |\bu_d|^2 \one_\omega\du x\right|\\
			& \le C_{1,1}\|\varphi^\infty - \varphi^{T_n} \|_{L^2}(K + T_n\|\bf_s\|_{\bL^2}^2 + \|\bu_0\|_{\bH}^2) + T_n \left|\int_\Omega \left( \frac{\varphi^\infty - \varphi^{T_n} }{2} \right) |\bu_d|^2 \one_\omega\du x\right|.
		\end{align*}
		\jshsb {After dividing  by $T_n$, we see that the first term tends to zero as $n\to\infty$, by the strong convergence of $\varphi^{T_n} \to \varphi^\infty$ in $L^2(\Omega)$.} 
		For the second term we use $|\bu_d|^2\one_\omega\in L^1(\Omega)$ and the convergence \eqref{linfconv} to show that in summary $\frac{1}{T_n}J_{1,1}\to 0$ as $n\to \infty$. 
		Similar arguments  show that $\frac{1}{T_n}J_{1,2}\to 0$ as $n\to \infty$. 
		
		To majorize $J_{1,3}$, we use H{\"o}lder inequality, the embedding $H^1\hookrightarrow L^4$, and the estimates  \eqref{energy:time} and \eqref{connect:ffs}, and proceed using a generic constant $C_{1,3}$.
		\begin{align*} 
			J_{1,3} & =  \left|\int_{\Omega_{T_n}}  {\mathcal{A}}_\varepsilon(\varphi^{T_n})(\varphi^\infty - \varphi^{T_n})\bu^{T_n}\cdot\bua^{T_n} \du x\du t \right|\\
			&\le \|\mathcal A_\varepsilon(\varphi^{T_n})(\varphi^\infty - \varphi^{T_n})\|_{L^2}\|\bu^{T_n}\|_{L^2(\bL^4)}\|\bua^{T_n}\|_{L^2(\bL^4)}\\
			&\le c_{\!\mathcal{A}}\|\varphi^\infty - \varphi^{T_n}\|_{L^q}\|\bu^{T_n}\|_{L^2(\bL^4)}\|\bua^{T_n}\|_{L^2(\bL^4)}\\
			& \le C_{1,3}\|\varphi^\infty - \varphi^{T_n}\|_{L^q}\|\bu^{T_n}\|_{L^2(\bV)}\|\bua^{T_n}\|_{L^2(\bV)}\\
			& \le C_{1,3}\|\varphi^\infty - \varphi^{T_n}\|_{L^q}( \|\bf\|_{L^2(\bL^2)} + \|\bu_0\|_{\bH} )\left\| \left( \frac{\varphi^{T_n} + 1}{2} \right) (\bu^{T_n}-\bu_d)\one_{\omega} + \alpha_\varepsilon(\varphi^{T_n})\bu^{T_n} \right\|_{L^2(\bL^2)} \\
			& \le C_{1,3}\|\varphi^\infty - \varphi^{T_n}\|_{L^q}( K + \sqrt{T_n}\|\bf_s\|_{\bL^2} + \|\bu_0\|_{\bH} ) 
			\Bigg[ \left( \frac{\|\varphi^{T_n}\|_{L^\infty} + 1}{2} \right) 
			\left(
			\|\bu^{T_n}\|_{L^2(\bL^2)}+ \sqrt{T_n}\|\bu_d\|_{\bL^2(\omega)} 
			\right)
			\\
			& \quad + \|\alpha_\varepsilon(\varphi^{T_n})\|_{L^\infty}\|\bu^{T_n} \|_{L^2(\bL^2)} \Bigg],
		\end{align*}
		where $q\ge 2$ is such that $1/2 = 1/q + 1/p$ and $p\ge 4$ is as in Assumption~\ref{porousassum}\ref{en:assump-alpha-ii}. Using the bounds \eqref{linfalpha}, 
		\eqref{bound:Phi} and 
		\eqref{energy:stat}, we arrive at the estimate
		\begin{align*}
			J_{1,3} \le C_{1,3}\|\varphi^\infty - \varphi^{T_n}\|_{L^q}(1 + \sqrt{T_n} + T_n )
		\end{align*}
		where $C_{1,3}>0$ is a constant not dependent on $T_n$. We thus have $\frac{1}{T_n}J_{1,3} \to 0$ as $n\to \infty$.
		Lastly,  we have  $J_{1,4}\to0$ as $n\to \infty$ by virtue of \eqref{lpconv}.

		Similarly, we can again use the first order optimality $J'_{T_n}(\varphi^{T_n})(\varphi^{T_n} - \varphi) \le 0$ with $\varphi = 2\varphi^{T_n} - \varphi^\infty\in \Phi_{ad}$ to majorize $\displaystyle \int_\Omega \nabla\varphi^{T_n}\cdot(\nabla\varphi^\infty - \nabla\varphi^{T_n}) \du x$ that goes to zero as $n\to \infty$. 
		Together,  
		$\left| \int_\Omega \nabla\varphi^{T_n}:(\nabla\varphi^\infty - \nabla\varphi^{T_n}) \du x \right| \to 0$ for $n\to \infty$ and  
		we  showed that $|\mathbb{E}_{\varepsilon}(\varphi^{T_n}) - \mathbb{E}_{\varepsilon}(\varphi^\infty) | \to 0$ as $n\to \infty$.

		In summary, we  conclude that 
		\begin{align*}
			|J_s(\varphi^s) - J_s(\varphi^{\infty})| \le  \lim_{n\to\infty} C\left(\frac{1}{T_n} + \frac{1}{\sqrt{T_n}}\right) + |J_s(\varphi^s) - J_{T_n}(\varphi^{T_n}) | =  0
		\end{align*}
		which completes the proof.    
	\end{proof}

	
	\section{Numerical implementation and examples}
	\label{sec:numerics}
	This section is dedicated to numerically illustrate the convergence of solutions proved in Section~\ref{sec:LongTimeBehavior}.
	The equations involved in the optimization problems are solved using finite element method with the aid of the free software \texttt{FreeFem++} \cite{hecht2012}. 
	
	\jshsr{To approximate the phase-field, $\mathbb{P}^1$ finite elements were used.} The state and adjoint equations are discretized using $\mathbb{P}^1$-bubble finite elements for the velocity component, while we use $\mathbb{P}^1$ elements for the pressure, i.e. we use the MINI-element.
	To solve the time-dependent Navier--Stokes equations we utilize an Implicit/Explicit method: we uniformly partition the interval $[0,T]$ into $N_T$ subintervals $(t^{n-1},t^{n})$ and recursively solve the approximations $(\bu^n,p^n) \approx (\bu,p)(t^n)$ using the
    \ck{linear}
    difference equation
	\begin{align*}
		\left\{
		\begin{aligned}
			\frac{1}{\Delta t}(\bu^n - \bu^{n-1}) + \alpha_\varepsilon(\varphi)\bu^n - \mu \Delta\bu^{n} + (\bu^{n-1}\cdot\nabla)\bu^n + \nabla p^n & = \bf^n &&\text{ in }\Omega,\\
			\dive\bu^n & = 0 &&\text{in }\Omega,\\
			\bu^n & = \bg^n &&\text{on }\partial\Omega,      
		\end{aligned}
		\right.
	\end{align*}
	with $\bu^0 = \bu_0$ and $\Delta t = t^{n} - t^{n-1}$. This method is a stable time-discretization method which in fact satisfies
	\begin{align*}
		\frac{1}{\Delta t}\|\bu^n\|_{\bH} + \frac{\mu}{2}\|\bu^n\|_{\bV} \le
		c\left( \|\bu_0\|_{\bH} + \|\bf\|_{L^2(\bL^2)} + \|\bg\|_{L^2(\bL^2(\partial\Omega))} \right)
	\end{align*}
	where $c>0$ continuously depends only on $\frac{1}{\mu\Delta t}$, $c_P$, and the continuity constant of the trace operator $\bH^1(\Omega)\hookrightarrow \bH^{1/2}(\partial\Omega)$. 
	An appropriate adjoint equation is also formulated that corresponds to the optimization problem and the semi-discretization above.
	
	The stationary problem, on the other hand, is solved using Newton's method for the nonlinear Navier--Stokes equations. The routine is terminated when the $\bV$-norm of the difference between two consecutive approximations reaches a given tolerance $\texttt{tol}>0$.
	
	
	To solve the discretized versions of the optimization problems \eqref{reduced:objphasestat} and \eqref{reduced:objphasetime}, we use the \enquote{variable metric projection-type} (VMPT) method in order to take account of the fact that the Fr{\'e}chet derivative of $J_T$ and $J_s$ is defined on $\mathbb{X}:=H^1(\Omega)\cap L^\infty(\Omega)$.
	
	\jshsr{In each optimization iteration, the VMPT method calculates suitable descent directions by solving a projection type subproblem. In this way a descent direction is found, that would correspond to the gradient in case of a Hilbert space setting.
	In our case these subproblems are linear-quadratic with bound constraints and thus are solved by
	a primal-dual active set (PDAS) method. We use the VMPT method from \cite{blank2017}, while the PDAS method is inspired by  \cite{ito2004}. We use the VMPT method for both, the time-dependent problem \eqref{reduced:objphasetime} and the stationary problem \eqref{reduced:objphasestat}. It is terminated when 
	the $H^1$-norm of the difference between two consecutive iterates reaches a given tolerance $\texttt{tol} > 0$.}

    \jshsr{
	For both, Newton's method and VMPT, we use $\texttt{tol} = 1 \times 10^{-6}$. All linear systems of equations resulting from finite element discretisations are solved with the sparse linear solver UMFPACK \cite{Davis.2004}.}

	\subsection{Set-up of the numerical example}
	
	To illustrate the convergence proven in the previous section, 
	we consider the problem of tracking the velocity field inside a channel that contains an obstacle.
	
	Let $\Omega = (0,3)\times (0,1)$ which we discretize with a  uniform structured triangular mesh from a $600\times 200$ grid. 
	To have a physically relevant example, instead of considering a homogeneous Dirichlet condition, we consider Dirichlet data $\bg$ and $\bg_s$ that have the behavior of inflow and outflow on the left and right borders of the domain, respectively denoted as $\Gamma_{i}$ and $\Gamma_{o}$ (cf. Figure~\ref{fig:num:sketchSetup}). In this case, we take zero values for the external forces $\bf_s$ and $\bf$. 
	We mention that even though the set-up of the flow in this numerical example differs from that in the analysis before, we can expect a similar convergence by lifting the Navier--Stokes equations with the extension of the Dirichlet data by virtue of the surjectivity of the trace function.
	
	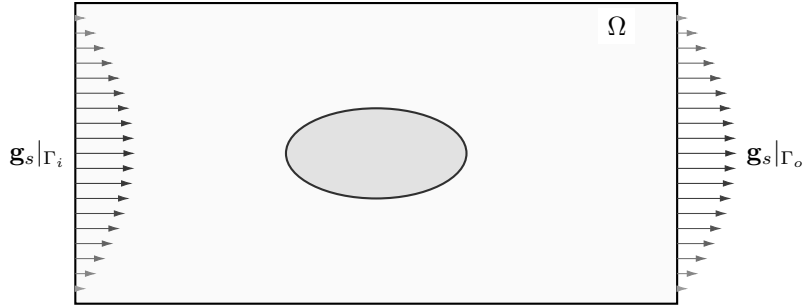
\begin{figure}
		\centering


\colorlet{clSlow}{black!30}
\colorlet{clFast}{black!80}

\colorlet{clEllipse}{black!40}

\begin{tikzpicture}[scale=4]

\pgfmathsetmacro{\Lx}{2}
\pgfmathsetmacro{\Ly}{1}


    
        

\draw[thick,fill=black!2] (0,0) rectangle (\Lx,\Ly);

\fill[clEllipse, opacity=0.25] (\Lx/2,\Ly/2) ellipse (0.3 and 0.15);          
\draw[black, thick,opacity=0.8] (\Lx/2,\Ly/2) ellipse (0.3 and 0.15);          

\foreach \y in {0.05,0.1,...,0.96} 
{
    \pgfmathsetmacro{\vx}{0.8*\y*(1-\y)}

    \pgfmathsetmacro{\pct}{100*(\y*(1-\y)*4}

    \draw[clFast!\pct!clSlow,-{Latex[length=1.6mm, width=0.9mm]}] (0, \y) -- (\vx, \y);
    
}
\node[left] at (0,0.5) {$\mathbf{g}_s|_{\Gamma_{i}}$};

\foreach \y in {0.05,0.1,...,0.96} 
{
    \pgfmathsetmacro{\vx}{0.8*\y*(1-\y)}
    \pgfmathsetmacro{\pct}{100*(\y*(1-\y)*4}

    \draw[clFast!\pct!clSlow,-{Latex[length=1.6mm, width=0.9mm]}] (\Lx, \y) -- (\Lx+\vx, \y);
}
\node[right] at (\Lx+0.2,0.5) {$\mathbf{g}_s|_{\Gamma_{o}}$};

    


    





\node[fill = white] at (\Lx-0.2, \Ly-0.075) {$\Omega$};

\end{tikzpicture}
		\caption{Geometric setup of the numerical example. We consider a flow along a channel with Poiseuille-type Dirichlet data $\bg_s$ on $\Gamma_{i}$ and $\Gamma_{o}$ and an enclosed obstacle.}
		\label{fig:num:sketchSetup}
	\end{figure}
	
	As Dirichlet boundary data we use $\bg_s(x_1,x_2) = (0,1-4(x_2-0.5)^2)^\top$ 
	and $\bg(t) = (1-e^{-30t})\bg_s$.
	For the time-dependent setting we use $\bu_0 = 0$.
	The interpolation functions $\alpha_\varepsilon$ and $\alphat_\varepsilon$, are chosen as
	\begin{align*}
		\alpha_\varepsilon(\varphi) =\frac{100  \overline{\alpha}}{2\varepsilon}(1-\varphi)
		\text{ and }
		\alphat_\varepsilon(\varphi) =\frac{\overline{\alpha}}{2\varepsilon}(1-\varphi).
	\end{align*}
	The rest of the parameters are stated in Table~\ref{tab:params}. For the tracking term, we use full observation, i.e. $\omega = \Omega$, for simplicity.
	\begin{table}
		\centering 
		\begin{tabular}{ccccccc}
			\toprule
			$\overline{\alpha}$  & $\mu$ & $\varepsilon$ & $\gamma$ \\
			\midrule
			$0.1$  & $0.5$ &  $0.0075$ & $0.005$\\
			\bottomrule
		\end{tabular}
		
		\caption{Parameter values for the numerical setup.}
		\label{tab:params}
	\end{table}

	\paragraph*{Definition of the target profile $\bu_d$.}
	To define the target profile $\bu_d$, we solve the stationary Navier--Stokes equations \eqref{poreq:statNS} with a given  phase field $\varphi_d$ defined as 
	\begin{align*}
		&\varphi_d(x) = -\Phi_0\left(\frac{1}{\varepsilon}\left( 1 -\sqrt{30(x_1-x_{c,1})^2 + 80(x_2-x_{c,2})^2} \right)  \right),\\ 
		&\text{ where }\Phi_0 = 
		\left\{ \begin{aligned}
			&\sin(z) &&\text{if }|z|\le \frac{\pi}{2},\\
			&\mathrm{sgn}(z) &&\text{otherwise},
		\end{aligned}
		\right.
	\end{align*}
	with $x_c = (1.5,0.5)^\top$.
	Here $\Phi_0$ is an approximation of a phase field obtained from the double-obstacle potential, see e.g.  \cite[Sec.~4.3.3]{Abels2012}.
	Note that $\varphi_d$ is a phase-field whose zero level line is
	the ellipse centered at $x_c$ with half axis of length $\sqrt{30}^{-1}$ and $\sqrt{80}^{-1}$
	that satisfies $$30(x_1-1.5)^2 + 80(x_2-0.5)^2 = 1.$$ 
	To reach an obstacle $\{x\in\Omega: \varphi_d(x) = -1 \}$ that is \textit{almost} impermeable,
	we also use a different interpolation $\alpha_\varepsilon$ of the porous media, given by $\alpha_{\varepsilon}(\varphi_d) = \frac{500\overline{\alpha}}{\varepsilon}(1-\varphi_d)$.
	We show the target profile $\bu_d$ together with the zero level line of $\varphi^d$ in Figure~\ref{fig:numstat:opt}.
	
	\begin{remark}
		We stress that even if we would consider the same permeability for $\bu_d$ and the optimization process, $\varphi_d$ and thus $\bu_d$ is not reachable by solving \eqref{reduced:objphasestat}
		or \eqref{reduced:objphasetime} since $J_s$ and $J_T$ contain regularization terms that have to be taken into account during minimization.
	\end{remark}

	\subsection{Solution to the stationary problem}

	\jshsr{We begin our numerical study by addressing the stationary problem \eqref{reduced:objphasestat}. The optimization is initialized with the phase-field $\varphi^s \equiv 1$, 
    from which the VMPT method terminates after 37  iterations.}

	\begin{figure}
		\centering
		\includegraphics[width=0.9\textwidth]{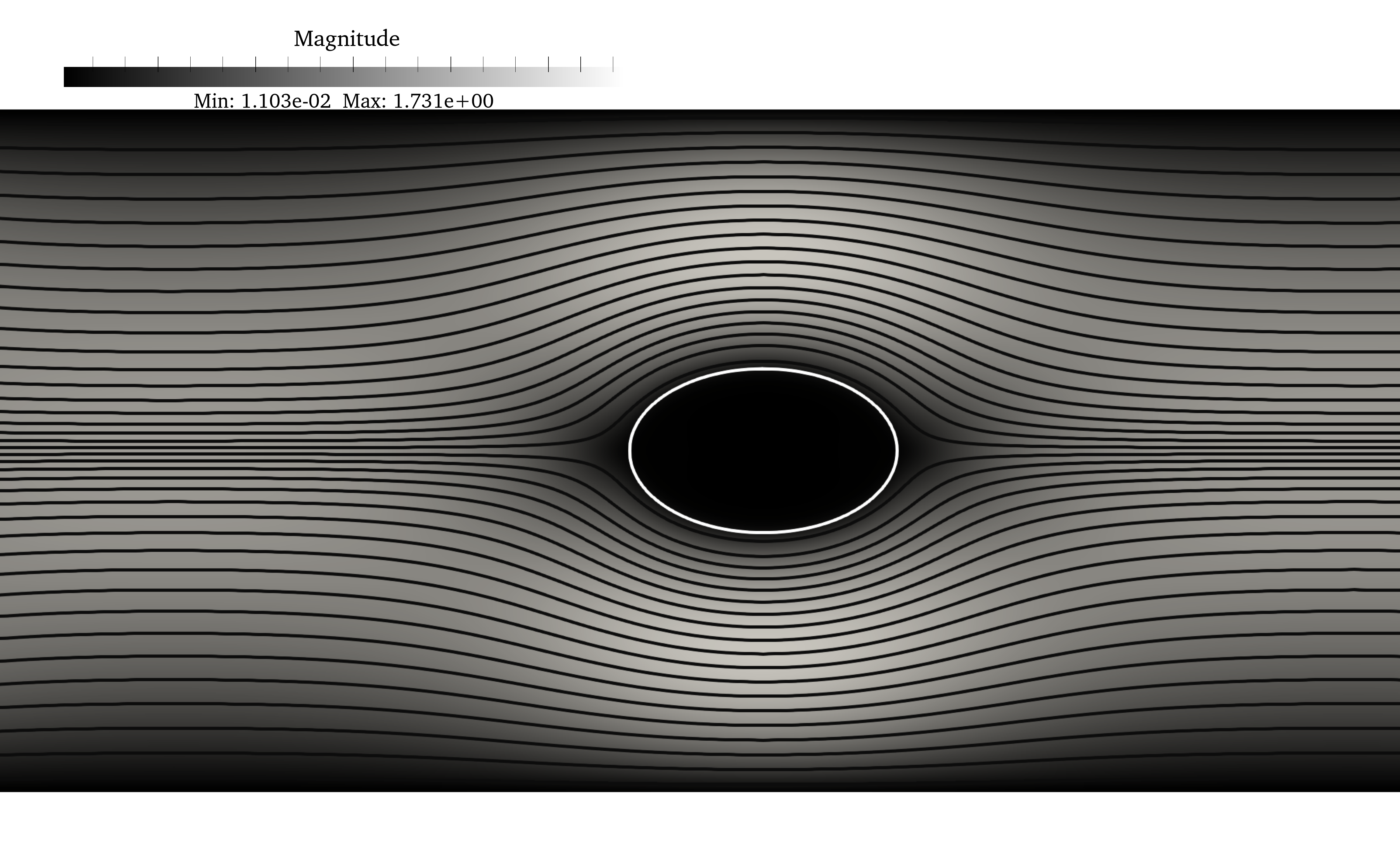}
		\caption{
			This image shows the desired velocity field $\bu_d$ (in magnitude and by streamlines) together with the zero-level set of $\varphi_d$. Despite the porous media approximation of the obstacle, we observe a substantial decrease of velocity inside the region $\{x\in\Omega: \varphi(x) \le 0\}$.
		}
		\label{fig:numstat:opt}
	\end{figure}
	
	In Figure~\ref{fig:numstat:phi-n} we show the approximate solution to \eqref{reduced:objphasestat} after $1,8,15,22,29,37$ iterations, where the result after 37 iterations is the \jshsr{obtained optimized shape}.
	In all cases we restrict to the subdomain $(0.9,2.0)\times (0.225,0.775) \subset \Omega$ that contains the present approximation of the region $\{x\in\Omega: -1\le \varphi^s(x) < 1 \} $.
	We observe that after the first step of the optimization procedure we find a reasonable guess for the location of the optimal topology and that this is subsequently slowly reduced to the optimal shape. We also note that the obtained optimal shape almost coincides with the covertices of the ellipse describing the zero-level set of $\varphi_d$, while a substantial difference is noticed at the vertices of the ellipse. In fact, the difference may be attributed to the fact that the target velocity $\bu_d$ is significantly low around the vertices.
	\begin{figure}
		\centering
		\includegraphics[width=0.3\textwidth]{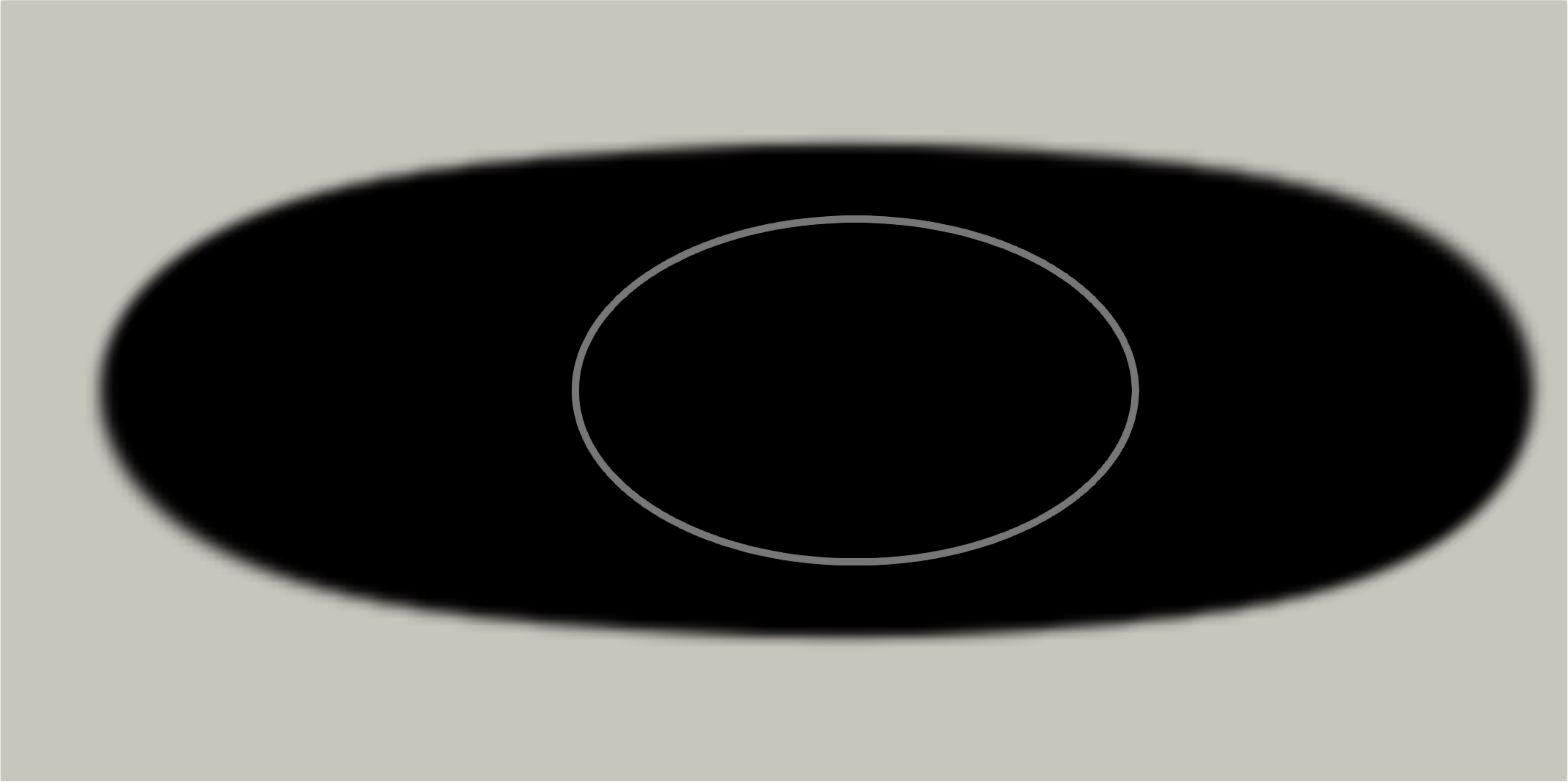}\hspace{1em}%
		\includegraphics[width=0.3\textwidth]{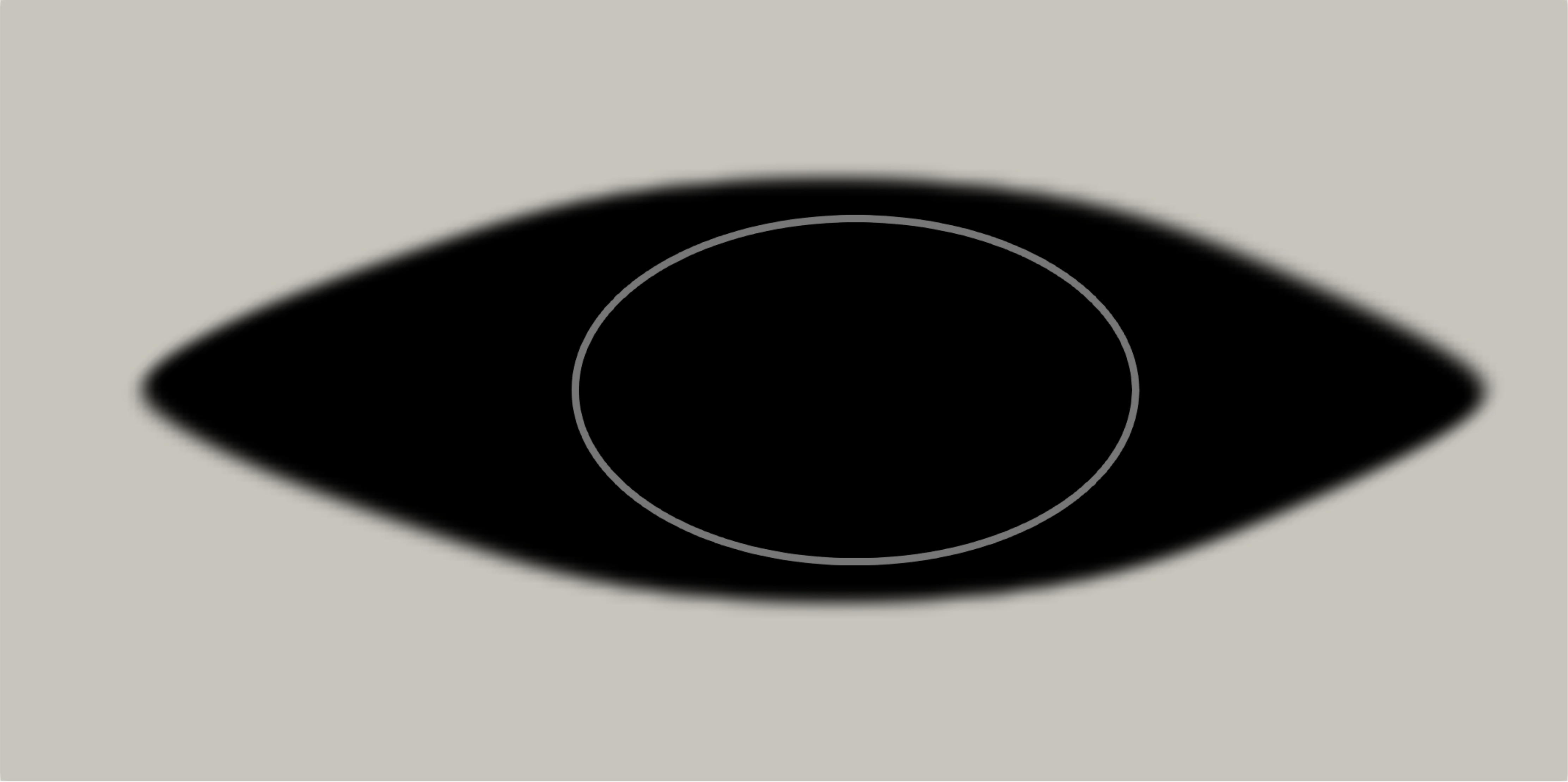}\hspace{1em}%
		\includegraphics[width=0.3\textwidth]{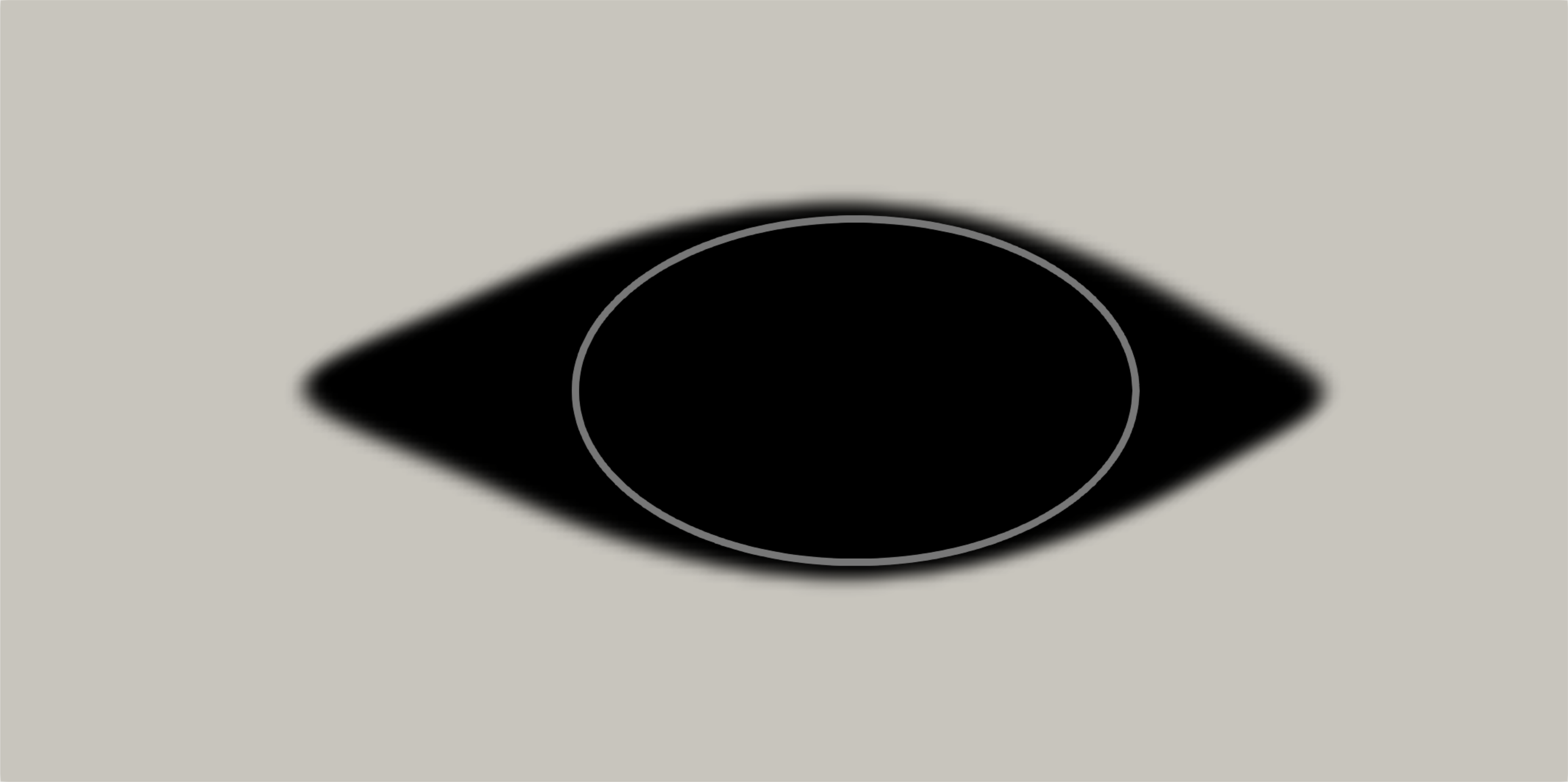}\\[1em]		 \includegraphics[width=0.3\textwidth]{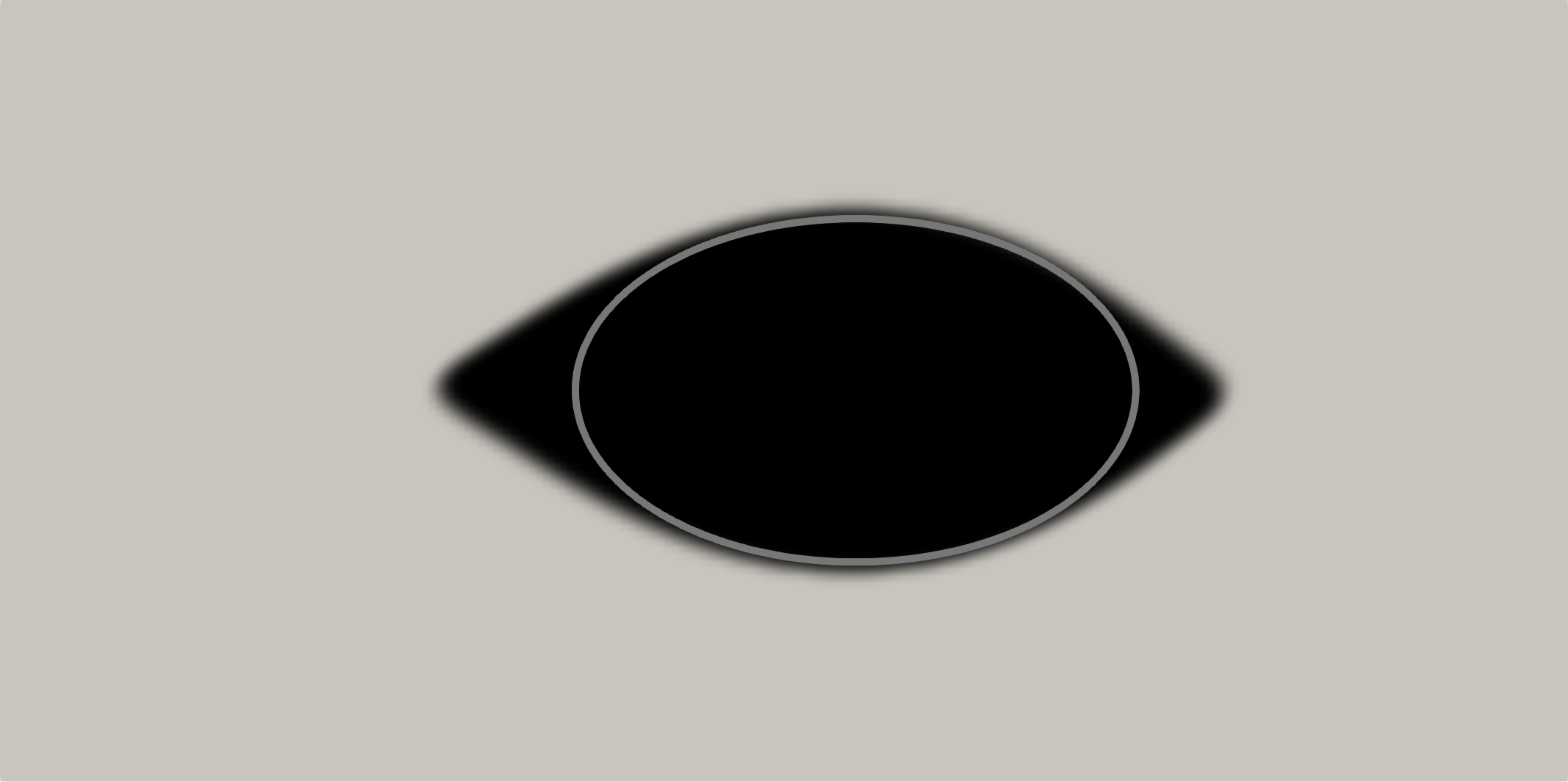}\hspace{1em}%
		\includegraphics[width=0.3\textwidth]{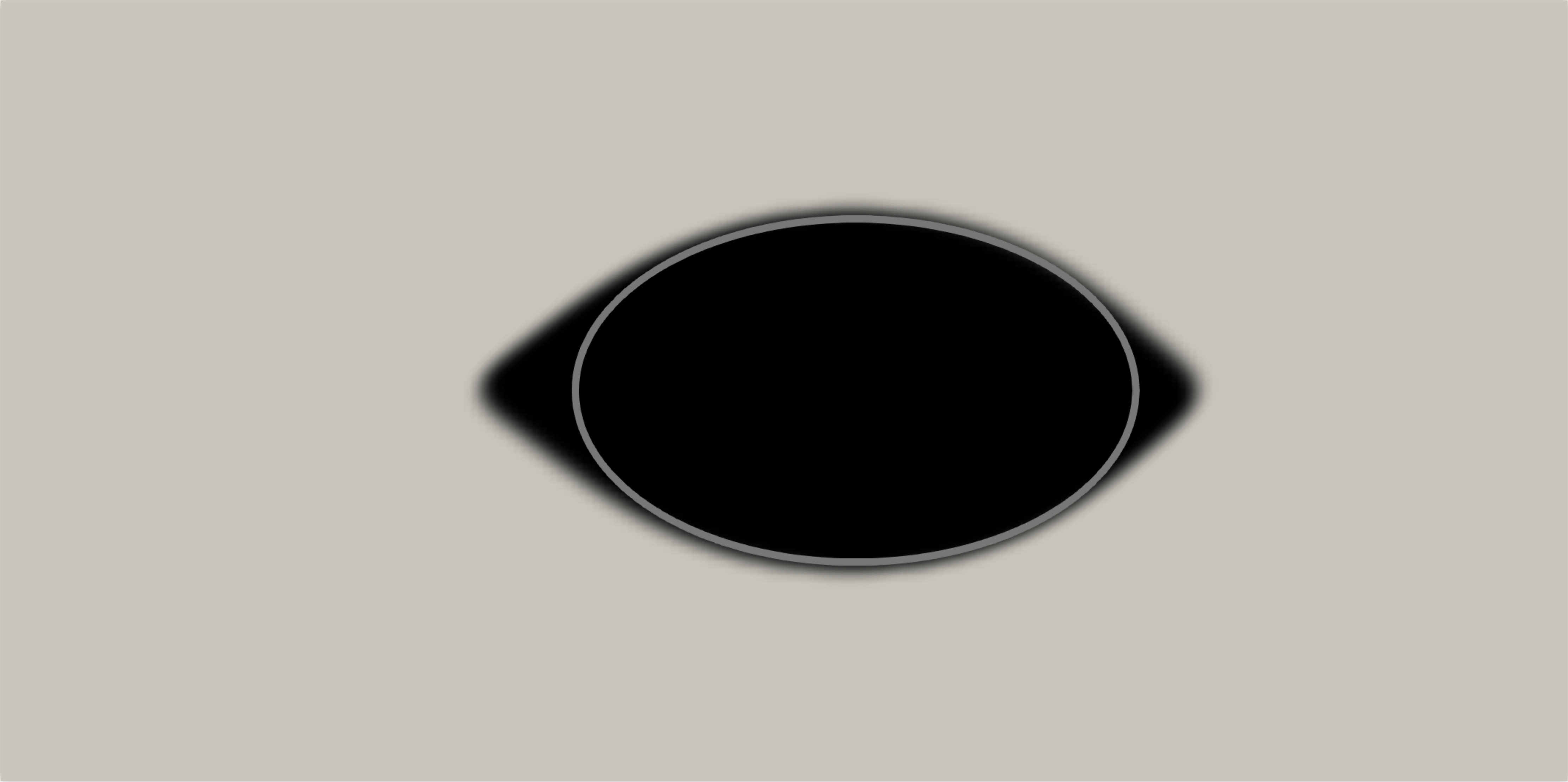}\hspace{1em}%
		\includegraphics[width=0.3\textwidth]{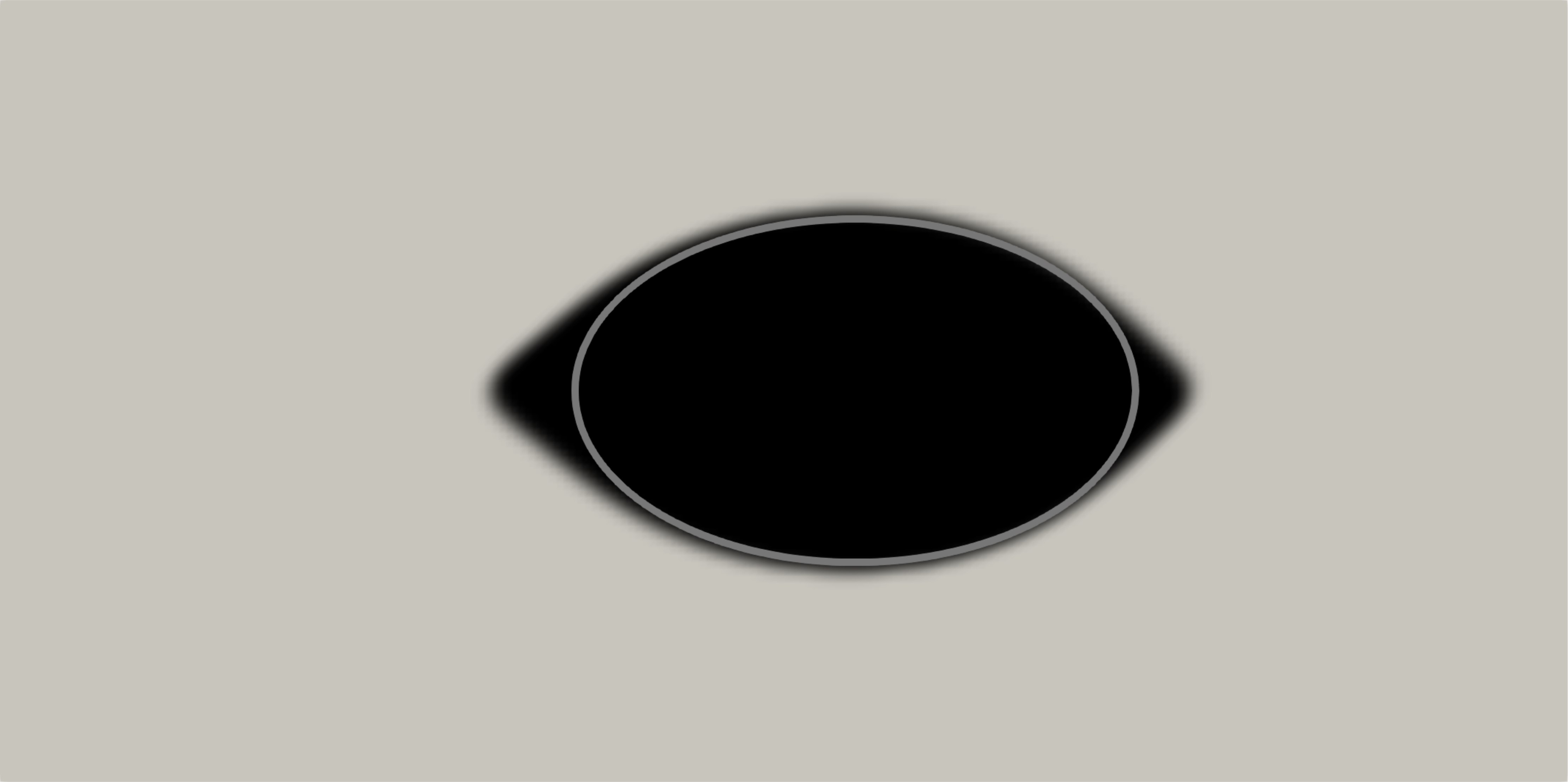}
		\caption{We show the approximate design function at iteration number $1,8,15,22,29,37$ of the optimization process from top left to bottom right, where the 37$^{th}$ iterate corresponds to the optimal shape. The figures also include the zero-level set of $\varphi_d$, for comparison. The images are cropped  to the subset $(0.9,2.0)\times (0.225,0.775)$ of $\Omega$. We observe that after the first iteration we achieve a reasonable guess for the location and the optimal topology which subsequently shrink towards the optimal shape, which also converges towards the \textit{target} phase-field $\varphi_d$.}
		\label{fig:numstat:phi-n}
	\end{figure}

	We also investigate the evolution of the objective function $J_s$ during the optimization process.
	In Figure~\ref{fig:opt-trend} we show the values of $J_s(\varphi^s)$ over the iteration number $n$, where $\varphi^s_n$ denotes the approximation to $\varphi^s$ after $n$ iterations. Additionally we show the three terms that contribute to $J_s(\varphi^s)$, 
	namely:
	\begin{enumerate}[label=(\roman*)]
		\item the tracking error $\displaystyle T_s(\varphi^s) = \int_\Omega \frac{\varphi^s+1}{2}|\bv - \bu_d|^2\dx$;
		\item the porous media penalty $\displaystyle P_s(\varphi^s) = \int_\Omega \beta_\epsilon(\varphi^s) |\bu|^2\dx$; and
		\item the Ginzburg--Landau energy $\displaystyle \gamma\mathbb E_\varepsilon(\varphi^s) = \frac{\gamma}{2c_0}\int_\Omega \frac{\epsilon}{2}|\nabla \varphi^s|^2 + \frac{1}{\epsilon}\Psi(\varphi^s)\dx$.
	\end{enumerate}
	
	We observe the typically hockey stick like behavior of steepest descent methods.
	Moreover the tracking error is the term that is most significantly affected throughout the optimization process, while the regularization terms $\mathbb{E}_\varepsilon$ and $P_s$ are only mildly affected by the optimization.
	We stress that these terms can only vanish if the obstacle vanishes. 
	The Ginzburg--Landau energy $\mathbb{E}_\varepsilon$ approximates the perimeter of the obstacle and can thus only vanish if the obstacle vanishes. Meanwhile, the porous media regularization $P_s$ can only vanish if either there is no obstacle or no flow inside the porous medium inside the obstacle.

	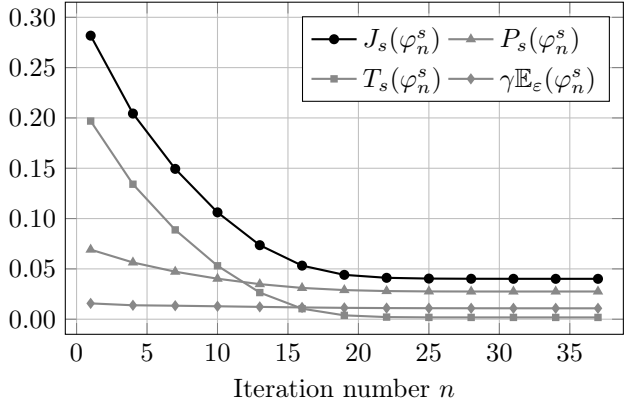
\begin{figure}
		\centering
		\definecolor{mygray}{HTML}{888888}%

\begin{tikzpicture}
\pgfmathsetmacro{\DataSpacing}{3}

\begin{axis}[
    width=9cm,
    height=6cm,
    xlabel={Iteration number $n$},
    ytick = {0,0.05,0.1,0.15,0.2,0.25,0.3},
    yticklabels = {$0.00$, $0.05$, $0.10$, $0.15$, $0.20$, $0.25$, $0.30$},
    grid=major,
    enlargelimits=0.05,
    legend columns=2,
    legend style={
        at={(0.7,0.7)},
        anchor=south,
        draw=black,
        fill=white,
        legend cell align=left,
        row sep=2pt
    },
    ymin=0, 
    ymax=0.3,
    ,each nth point=\DataSpacing
]

\addplot[
    color=black,
    mark=*,
    mark size=1.5pt,
    thick,
]   table[row sep=crcr]{%
1	0.281763\\
2	0.25222\\
3	0.226684\\
4	0.204374\\
5	0.184374\\
6	0.166104\\
7	0.14938\\
8	0.13384\\
9	0.119433\\
10	0.106179\\
11	0.0940434\\
12	0.0831703\\
13	0.0735658\\
14	0.0653944\\
15	0.0586338\\
16	0.053249\\
17	0.0491579\\
18	0.0461662\\
19	0.0440723\\
20	0.0426571\\
21	0.0417319\\
22	0.0411295\\
23	0.040756\\
24	0.0405203\\
25	0.0403719\\
26	0.0402819\\
27	0.0402237\\
28	0.0401881\\
29	0.0401634\\
30	0.0401489\\
31	0.0401399\\
32	0.040134\\
33	0.0401295\\
34	0.0401265\\
35	0.0401249\\
36	0.0401238\\
37	0.0401228\\
};
\addlegendentry{$J_{s}(\varphi_n^s)$}

\addplot[
    color=mygray,
    thick,
    mark=triangle*,
    mark size=1.5pt,
]   table[row sep=crcr]{%
1	0.0691396\\
2	0.0646797\\
3	0.0601757\\
4	0.0563938\\
5	0.0530196\\
6	0.0499303\\
7	0.0472095\\
8	0.0446994\\
9	0.0423493\\
10	0.0402458\\
11	0.0382528\\
12	0.0364937\\
13	0.034843\\
14	0.0334355\\
15	0.0321797\\
16	0.0310855\\
17	0.0302138\\
18	0.0295025\\
19	0.0289455\\
20	0.0285132\\
21	0.0281963\\
22	0.0279635\\
23	0.0278007\\
24	0.027688\\
25	0.0276128\\
26	0.0275679\\
27	0.0275401\\
28	0.0275247\\
29	0.0275162\\
30	0.0275143\\
31	0.0275152\\
32	0.0275175\\
33	0.0275208\\
34	0.0275247\\
35	0.0275285\\
36	0.027532\\
37	0.027535\\
};
\addlegendentry{$P_s(\varphi_n^s)$}

\addplot[
    color=mygray,
    thick,
    mark=square*,
    mark size=1.0pt
] table[row sep=crcr]{%
1	0.196904\\
2	0.173468\\
3	0.152515\\
4	0.134138\\
5	0.117654\\
6	0.102588\\
7	0.0887612\\
8	0.0759167\\
9	0.0640318\\
10	0.0530901\\
11	0.0431405\\
12	0.0342347\\
13	0.0264751\\
14	0.019919\\
15	0.0145886\\
16	0.0104662\\
17	0.00741157\\
18	0.00526366\\
19	0.00383807\\
20	0.00294493\\
21	0.0024108\\
22	0.00210358\\
23	0.00194178\\
24	0.00185841\\
25	0.00181705\\
26	0.00179828\\
27	0.00178949\\
28	0.00178568\\
29	0.00178334\\
30	0.00178182\\
31	0.0017803\\
32	0.0017788\\
33	0.00177723\\
34	0.00177544\\
35	0.00177367\\
36	0.00177179\\
37	0.00176991\\
};
\addlegendentry{$T_s(\varphi_n^s)$}


\addplot[
    color=mygray,
    thick,
    mark=diamond*,
    mark size=1.5pt,
]  table[row sep=crcr]{%
1	0.0157188\\
2	0.01407255\\
3	0.0139936\\
4	0.0138418\\
5	0.0137007\\
6	0.01358555\\
7	0.01340885\\
8	0.01322355\\
9	0.01305185\\
10	0.0128435\\
11	0.01265005\\
12	0.01244185\\
13	0.01224765\\
14	0.01203985\\
15	0.01186545\\
16	0.0116974\\
17	0.0115325\\
18	0.01139995\\
19	0.01128875\\
20	0.01119895\\
21	0.01112485\\
22	0.01106245\\
23	0.01101355\\
24	0.0109739\\
25	0.010942\\
26	0.01091565\\
27	0.0108941\\
28	0.0108777\\
29	0.01086395\\
30	0.01085275\\
31	0.01084445\\
32	0.0108377\\
33	0.01083145\\
34	0.01082635\\
35	0.0108227\\
36	0.01082\\
37	0.01081795\\
};
\addlegendentry{$\gamma\mathbb{E}_\varepsilon(\varphi_n^s)$}


\end{axis}

\end{tikzpicture}
		\caption{ The figure shows the evolution of $J_s$ over the optimization process. The optimization stops after 37 iterations by reaching the stopping criterion. Moreover, we show the evolution of the individual components of $J_s$.
			Furthermore, we observe a substantial decrease on the tracking part $T_s$, while the regularization terms $P_s$ and $\gamma\mathbb{E}_\varepsilon$ exhibit a minimal decrease and equilibrates early on during the optimization process.}
		\label{fig:opt-trend}
	\end{figure}

	\subsection{Minimizers of $J_T$ and comparison to $\varphi^s$ }

	In this section, we present the results concerning the time-dependent problem. We find the minimizer of $J_T$ for varying values of the final time, i.e., $T \in \{0.5,1,2,4,8,16\}$. To resolve the time-dependent problems, we use $\varphi^T_0 \equiv 1$ as initial guess for the case $T = 0.5$, and use the optimal solution for the case $T=2^{n-1}$, $n=1,2,3,4$,  as initial data for finding the optimal phase-field for the case $T = 2^n$.

	To back up the theoretical results obtained in Section~\ref{sec:LongTimeBehavior}, we first look at the gap $|J_s(\varphi^s) - J_T(\varphi^T)|$ against the time-horizon, in log-log scale. We see in Figure~\ref{fig:time_comp} that the slope of $|J_s(\varphi^s) - J_T(\varphi^T)|$, \jshsb{in our numerical range of $(0,16]$, follows $\mathcal{O}(T^{-1})$. In our numerical framework, we assumed zero external forcing for both the stationary and time-dependent Navier-Stokes equations. Instead, we utilize Dirichlet data to prescribe in/out-flow conditions, thereby modeling physically relevant phenomena. By classical lifting those boundary conditions may be interpreted as a volume force again denoted by $\bf_s$. We assume that the analytical results in our setting remain largely invariant under this lifting process.
    }
\jshsb{
    \begin{remark}
        In our numerical results we observe a decay rate of order $\mathcal O(T^{-1})$. 
        This can be attributed to different sizes of the constants $C_1$ and $C_2$ in the estimate of Theorem \ref{thm:asymp} and the time interval  $[0,16]$ considered. 
    If one would use the Poiseuille flow which determines our boundary conditions as classical lifting function, we would obtain a volume force $\bf_s$ for which $\|\bf_s\|_{L_2}\sim 3$.
    Since $\|\bu_d\|_{L^2}\sim 10$ we have $C_1\approx 4C_2 = \sqrt{T}C_2$ for $T=16$, so that on the time interval $[0,16]$, we expect to observe the decay rate $T^{-1}$.
    This would then explain the rate observed in our numerical simulations. 
    However, in our numerical implementation, we use a practical method of lifting the boundary conditions, whereby the respective values are incorporated directly into the discrete systems that appear. Nevertheless, we would expect to see similar numerical behavior in this case.
        %
    \end{remark}
}
	
	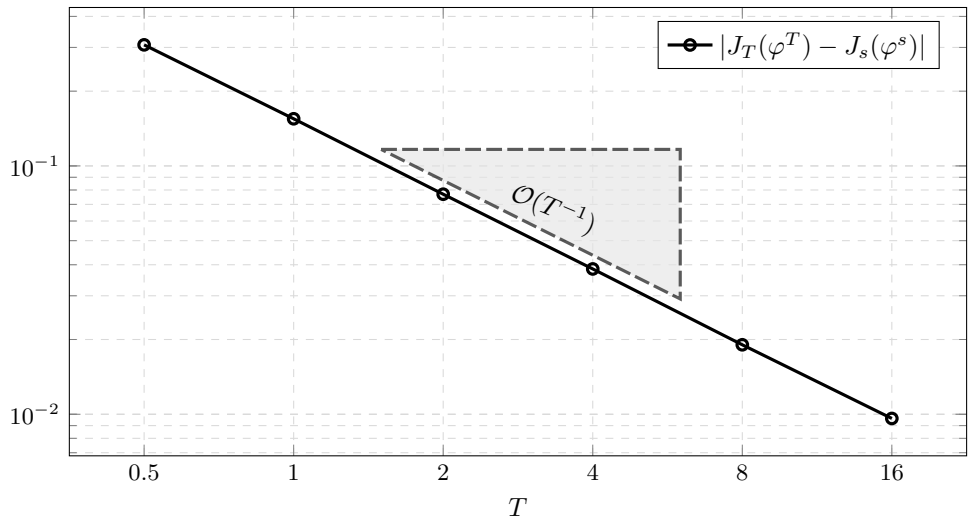
\begin{figure}
		\centering
		\begin{tikzpicture}
  \begin{loglogaxis}[
      width=12cm,
      height=6cm,
      scale only axis,
      xtick={0.5,1,2,4,8,16},
      xticklabels ={$0.5$,$1$,$2$,$4$,$8$,$16$},
      xlabel={$T$},
      grid=both,
      grid style={dashed,gray!30},
      legend pos=north east,
      legend cell align={left},
      tick label style={font=\small},
      title style={yshift=1ex},
      legend style={cells={anchor=west}, fill=white},
      mark options={solid}
    ]

    \coordinate (A) at (axis cs:1.5, 0.11654330513 );       
    \coordinate (B) at (axis cs:6, 0.11654330513 );       
    \coordinate (C) at (axis cs:6, 0.02913582628 );       

    \draw[fill=gray!20,mark size = 2, domain=1:4, samples=5, line width=1.2pt, dash pattern=on 6pt off 2pt, opacity=0.65]
      (A) -- (B) -- (C) -- cycle;

    \addplot+[black, line width=1.2pt, mark=o, mark size=2pt]
      coordinates {
        (0.5, {0.3473868591 - 0.0401228})
        (1,   {0.1949377577 - 0.0401228})
        (2,   {0.117066 - 0.0401228})
        (4,     {0.07860091931 - 0.0401228})
        (8, {0.059166251 - 0.0401228})
        (16, {0.04974099986 - 0.0401228})      
      };
        \addlegendentry{$|J_T(\varphi^T) - J_s(\varphi^s)|$}

    \node[anchor=east,rotate=-23] at (axis cs:4.2,0.055) {$\mathcal{O}(T^{-1})$};


  \end{loglogaxis}
\end{tikzpicture}
		\caption{The figure shows in log-scale the gap $|J_T(\varphi^T)-J_s(\varphi^s)|$ against the time-horizon  $T\in\{0.5,1,2,4,8,16\}$. We observe that the  gap follows the same slope as $\mathcal{O}(T^{-1})$ which agrees with the theoretical result \eqref{goal}.}
		\label{fig:time_comp}
	\end{figure}

	On the other hand, in Figure~\ref{fig:compstat} we present the zero-level set of the optimal solutions $\varphi^T$ for $T\in\{0.5,1,2,4,8,16\}$.  Figure~\ref{fig:compstat}(A) shows the zero-level set of the optimal solutions in comparison with the zero-level set of the optimal solution $\varphi^s$ to the stationary problem \eqref{reduced:objphasestat}. We observe that the solutions $\varphi^T$ move towards the solution $\varphi^s$ as we increase the final time $T$. To further corroborate such convergence, we also show the cross sections of the phase-fields through the plane passing through the points (1.5,0.6) and (1.5,0.65) in Figure~\ref{fig:compstat}(B), and through the plane passing through the points (1.7125,0.5) and (1.755,0.5) in Figure~\ref{fig:compstat}(C). This accumulation of the solutions illustrates the convergence established in \jshsb{\eqref{h1conv}, \eqref{lpconv}, \eqref{linfconv} and} Theorem~\ref{thm:phiInftyIsMin}.
		
	\begin{figure}
		\centering
		\begin{subfigure}[b]{0.85\textwidth}
			\centering
			\caption*{(A)}
			\includegraphics[width=\textwidth]{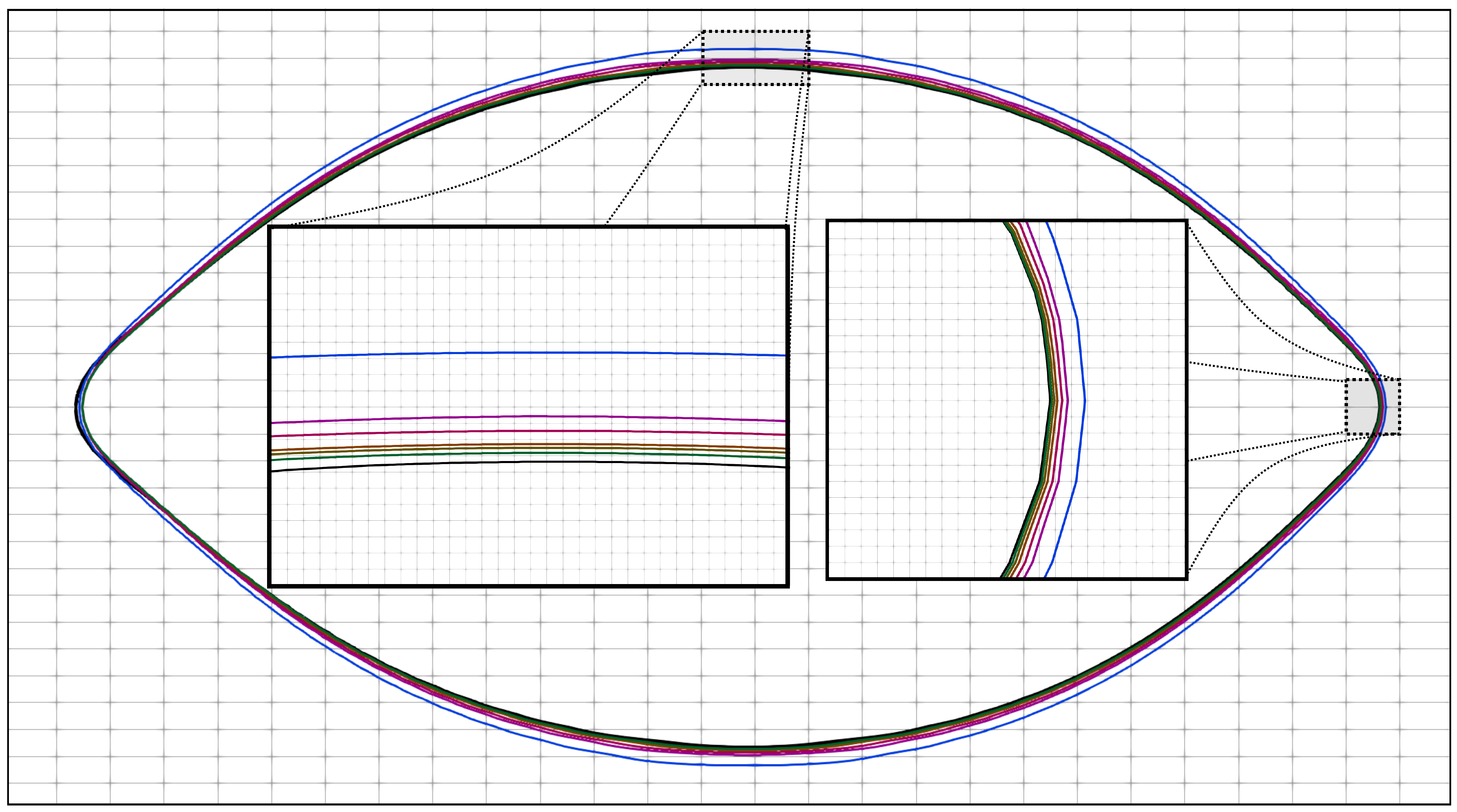}
		\end{subfigure}\\
		\begin{subfigure}[b]{0.45\textwidth}
			\centering
			\caption*{(B)}
			\includestandalone[width=\textwidth]{figures/standalone/cross_vert} 
		\end{subfigure}
		\begin{subfigure}[b]{0.45\textwidth}
			\centering
			\caption*{(C)}
			\includestandalone[width=\textwidth]{figures/standalone/cross_hor}
		\end{subfigure}
		\caption{We illustrate the qualitative behavior of the optimal solutions $\varphi^T$ for values $T = 0.5,1,2,\ldots,16$ and with comparison to the solution $\varphi^s$ to the stationary problem. (A) The zero-level sets of the solutions $\varphi^T$, for values $T = 0.5,1,2,\ldots,16$, and $\varphi^s$ are shown, we can also observe an accumulation of the level sets as we increase the time horizon $T$. (B) We plot the cross-section of the phase-fields on the line passing through the points $(1.5,0.6)$ and $(1.5,0.65)$. (C) We plot the cross-section of the phase-fields on the line passing through the points $(1.7125,0.5)$ and $(1.755,0.5)$. Figures (B) and (C) show how the solutions $\varphi^T$ accumulate toward the solution $\varphi^s$ as we increase the time-horizon $T$ illustrating the convergence of the optimal solutions of the time-dependent problem towards a minimizer of $J_s$ as shown in Theorem~\ref{thm:phiInftyIsMin}.}
		\label{fig:compstat}
	\end{figure}

	\section{Conclusion}
	
	In this work, we investigated the long-time behavior of solutions to a shape and topology optimization problem governed by the time-dependent Navier–-Stokes equations. The topology of the fluid domain was represented by a stationary phase-field variable acting as a smooth indicator function, while the fluid equations were formulated on a fixed hold-all domain using a porous media approximation. This setting extends earlier phase-field-based approaches for stationary flows to a fully time-dependent framework.
	
	The main analytical result shows that, as the time horizon tends to infinity, global minimizers of the time-dependent optimization problem converge in objective value to a global minimizer of the corresponding stationary problem. Moreover, an explicit convergence rate with respect to the time horizon was derived. 
	
	The theoretical findings are complemented by numerical investigations, which confirm the predicted convergence behavior and illustrate the practical relevance of the asymptotic analysis. From a numerical point of view, the results suggest that stationary or long-time-averaged models may serve as effective surrogates for large-horizon time-dependent optimization problems, thereby offering potential reductions in computational cost. Overall, this work advances the analytical and numerical understanding of phase-field-based shape optimization for unsteady fluid flows and provides a foundation for further developments at the interface of variational methods, optimal control, and numerical optimization.
	
	From a methodological perspective, the problem formulation exhibits a bilinear structure, as the design function enters the governing equations as a coefficient. The analysis therefore contributes to the understanding of asymptotic regimes in bilinear PDE-constrained optimization problems and provides a basis for investigating turnpike-type behavior in shape and topology optimization. The analytical tools developed herein are not limited to the specific tracking-type functional considered, but can be adapted to alternative objective functionals and to other evolutionary PDE models.

	\section*{Statements and Declarations.}
	
	\subsubsection*{Conflict of Interest.}
	The authors declare that they have no conflict of interest.
	
	\subsubsection*{Funding.}
	The authors acknowledge funding of the project \textit{Fluiddynamische Formoptimierung mit Phasenfeldern und Lipschitz-Methoden} by the German Research Foundation under the project number \href{https://gepris.dfg.de/gepris/projekt/543959359?language=en}{543959359}.
	
	\subsubsection*{Author Contribution.}
	(CRediT taxonomy)
	\begin{description}
		\item[M.H.]  Funding Acquisition,  Methodology, Project Administration, Resources, Supervision,  Writing - Original Draft, Writing - Review \& Editing
		\item[C.K.]  Funding Acquisition, Methodology, Project Administration, Resources, Supervision, Writing - Original Draft, Writing - Review \& Editing
		\item[J.S.H.S.]  Conceptualization, Data Curation, Methodology, Project Administration, Software, Visualization, Writing - Original Draft, Writing - Review \& Editing
	\end{description}

	\subsubsection*{Acknowledgement.}
	The authors acknowledge funding by the German Research Foundation within the project \textit{Fluiddynamische Formoptimierung mit Phasenfeldern und Lipschitz-Methoden} under the project number \href{https://gepris.dfg.de/gepris/projekt/543959359?language=en}{543959359}.

	
	
	\begingroup
	\printbibliography
	\endgroup
	
\end{document}